\documentclass[onefignum,onetabnum]{siamart171218}

\usepackage{lipsum}
\usepackage{amsfonts}
\usepackage{graphicx}
\usepackage{epstopdf}
\usepackage{algorithmic}
\ifpdf
  \DeclareGraphicsExtensions{.eps,.pdf,.png,.jpg}
\else
  \DeclareGraphicsExtensions{.eps}
\fi

\newsiamremark{remark}{Remark}
\newsiamremark{hypothesis}{Hypothesis}
\crefname{hypothesis}{Hypothesis}{Hypotheses}
\newsiamthm{claim}{Claim}

\headers{Greedy parameter selection for one-way equations}{M.K. Sleeman, and Tim Colonius}

\title{Greedy recursion parameter selection for one-way spatial integration of hyperbolic equations\thanks{Submitted to the editors October 27, 2025.
\funding{This work was funded by The Boeing Company through the Strategic Research and Development Relationship
Agreement CT-BA-GTA-1.}}}

\author{Michael K.\ Sleeman\thanks{California Institute of Technology, Pasadena, CA 
  (\email{msleeman@caltech.edu}).}
\and Tim Colonius\footnotemark[1]}

\usepackage{amsopn}
\DeclareMathOperator{\diag}{diag}

\makeatletter
\newcommand*{\addFileDependency}[1]{
  \typeout{(#1)}
  \@addtofilelist{#1}
  \IfFileExists{#1}{}{\typeout{No file #1.}}
}
\makeatother

\ifpdf
\hypersetup{
  pdftitle={GreedyOWNS},
  pdfauthor={M.K. Sleeman, and T. Colonius}
}
\fi

\usepackage{bm}
\usepackage{subcaption}

\usepackage{booktabs}

\begin{document}

\maketitle

\begin{abstract}
Solutions to hyperbolic systems comprise waves propagating at finite speeds.
When wave propagation is predominantly unidirectional, one-way wave equations
can be used to evolve only the right-going solution by removing support for
left-going waves. The One-Way Navier-Stokes (OWNS) approach, which was
originally developed for systems of first-order hyperbolic equations, constructs
one-way approximations to the linearized Navier-Stokes equations using a
recursive filter to remove left-going waves. The computational cost scales with
the number of recursion parameters, which must be carefully chosen to ensure
accuracy and stability of the resulting one-way equation. Previous work has
chosen parameters based on heuristic estimates of key eigenvalues, which
requires trial-and-error tuning while also yielding slow error convergence. We
propose a greedy algorithm for automatic parameter selection, which we show
yields faster convergence and a net decrease in computational cost for linear
and nonlinear disturbance evolution in boundary-layer flows. We review the OWNS
projection (OWNS-P) and recursive (OWNS-R) methods, comparing their convergence
properties, and show through our numerical analysis and experiments that OWNS-P
yields superior convergence and stability properties. Although we demonstrate
the method for Navier-Stokes equations, we perform our analyses on systems of
linear first-order hyperbolic equations and emphasize that the greedy algorithm
is applicable to such systems.
\end{abstract}



\section{Introduction}

Wave propagation can be modeled by (linear) systems of first-order hyperbolic
partial differential equations (PDEs), solvable either as time-domain initial
boundary value problems or frequency-domain boundary value problems. When wave
propagation is predominantly unidirectional, one-way wave equations can be used
to evolve the right-going solution by removing support for left-going waves.
Such equations can be obtained by factoring the dispersion relation in
Fourier–Laplace space into left- and right-going factors, and then retaining
only the right-going branch. Since the resulting equation contains a square root
of the Fourier-Laplace variables, the transformation back to physical space
results in a nonlocal integro-differential equation, which can be localized
using rational approximations of the square root~\cite{Lee_2000_Parabolic}.
Although these methods are accurate, well-posed, and efficient for simple wave
equations~\cite{Trefethen_1986_Parabolic,Halpern_1988_OneWay} such as the
equations for geophysical migration of seismic
waves~\cite{Claerbout_1976_Geophysics,Claerbout_1985_Geophysics} and underwater
acoustics~\cite{Collins_1989_Underwater,Jensen_1995_Underwater}, they can only
be applied in cases where the eigenvalues can be determined analytically, since
the dispersion relation must be factored analytically. Methods for one-way
marching that do not depend on this factorization have been developed by
Guddati~\cite{Guddati_2006_OneWay} for acoustic and elastic wave equations, and
by Towne and Colonius~\cite{Towne_2015_OWNS-O} for general systems of
first-order hyperbolic equations.

While the linearized Euler equations comprise a system of first-order hyperbolic
equations, the linearized Navier-Stokes do not. Nevertheless, the One-Way
Navier-Stokes (OWNS) framework generalizes the approach of Towne and Colonius to
the linearized Navier-Stokesequations~\cite{Towne_2022_OWNS-P,Kamal_2020_HOWNS,
Kamal_2021_OWNS_IO,Kamal_2022_HOWNS,Rigas_2017_OWNS_BL}, where it has been used
for linear disturbance evolution in turbulent jets and boundary layers. When
applied to the Navier-Stokes equations, the method is now referred to as the
OWNS outflow (OWNS-O) approach since it is based on work by Givoli and
Neta~\cite{Givoli_2003_NRBC} and Hagstrom and
Warburton~\cite{Hagstrom_2004_NRBC} for non-reflective boundary conditions
(NRBCs) at ``outflow'' boundaries. For brevity, we will use the OWNS label, but
we will perform our analysis on systems of linear first-order hyperbolic
equations, and refer to Towne et al.~\cite{Towne_2022_OWNS-P} for the details on
generalizing the method to the Navier-Stokes equations.

The \textit{exact} formulation of OWNS-O avoids factoring the dispersion
relation by (i) discretizing in the transverse directions to obtain a
semi-discrete ordinary differential equation (ODE) in the marching
direction~\eqref{eq:HyperbolicPDE_D_char_freq}, (ii) computing the
eigendecomposition of the discretized linear operator, $M$
in~\eqref{eq:HyperbolicPDE_D_char_freq}, (iii) using Briggs'
criterion~\cite{Briggs_1964_Electron} to classify the eigenvectors as left or
right-going, and (iv) using this information to construct a one-way equation for
the right-going waves. Since computing the eigendecomposition is computationally
expensive, the \textit{approximate} formulation instead uses a \textit{recursive
filter} to construct an approximate one-way equation, which converges to the
exact one-way equation if its \textit{recursion parameters} are well-chosen. If
the location (in the complex plane) and direction (left- or right-going) of all
eigenvalues are known, then it is always possible to choose convergent recursion
parameters. However, since the recursive filter is used to avoid computing the
eigendecomposition of $M$, convergent recursion parameters must be chosen
without complete knowledge of the eigenvalues of $M$. Towne and Colonius used
the eigenvalues of the Euler equations, linearized about a uniform flow (which
can be computed analytically), to inform recursion parameter selection for
nonuniform flows. Throughout this paper, we refer to this as \textit{heuristic}
recursion parameter selection.

OWNS-O can only be applied to homogeneous (unforced) equations, motivating the
development of the OWNS projection (OWNS-P)~\cite{Towne_2022_OWNS-P} and OWNS
recursive (OWNS-R)~\cite{Zhu_2021_OWNS-R} approaches, which can be applied to
inhomogeneous (forced) equations. We note that OWNS-P and OWNS-R use recursive
filters to obtain approximate one-way equations, which differ from each other
but are both approximations to the same \textit{exact} one-way equation (which
is also different from the \textit{exact} OWNS-O equation). We further note that
note that Rudel et al.~\cite{Rudel_2022_Bremmer} have also used recursive
filters to construct one-way equations for wave propagation in complex media,
which can also be applied to inhomogeneous equations, while Sleeman et
al.~\cite{Sleeman_2025_NOWNS_AIAAJ} have developed the nonlinear OWNS (NOWNS)
approach, based on OWNS-P, which enables nonlinear disturbance evolution in
boundary layers.

Although the \textit{heuristic} recursion parameter selection routine developed
by Towne and Colonius has proven effective for the Euler and Navier-Stokes
equations, it requires problem-specific parameter tuning (e.g., subsonic and
supersonic boundary layer flows require different parameters). As
noted above, the eigenvalues of $M$ can be used to construct convergent
recursion parameter sets. However, these sets are large, and smaller sets are
desirable for accuracy, stability, and efficiency. We therefore formulate a
\textit{subset selection} problem that, for a prescribed set size, selects a
subset of the eigenvalues of $M$ such that the OWNS approximation error is
minimized, enabling automatic recursion parameter selection. Subset selection is
NP-hard~\cite{Natarajan_1995_OptimalSubset}, so we solve the optimization
problem using a \textit{greedy} algorithm. We demonstrate for linear and
nonlinear disturbance evolution in subsonic and supersonic boundary-layer flows
(modeled using the Navier-Stokes equations) that the greedy approach yields
faster filter convergence than the heuristic approach, while also leading to a
reduced computational cost. We further show for nonlinear disturbance evolution
in a supersonic boundary layer flow that while the heuristic selection
contaminates the solution with numerical error, this problem is alleviated using
greedy selection.

We additionally analyze the convergence properties of the OWNS-P and OWNS-R
methods, showing that OWNS-P converges faster and that there exist
recursion parameter sets for which OWNS-P is fully converged while OWNS-R is
not. We demonstrate these properties numerically: our heuristic parameter
selection routine yields stable OWNS-P marches but unstable OWNS-R marches,
whereas greedy selection produces stable OWNS-R marches.

Section~\ref{sec:OWNS} reviews the OWNS-P and OWNS-R methods and analyzes their
stability and convergence properties, while Section~\ref{sec:greedy} presents
the greedy algorithm for recursion parameter selection. We present the
Navier-Stokes equations in~\ref{sec:NSE}, followed by convergence results at a
single station in  Section~\ref{sec:greedySingle}, and spatial marching results
in  Section~\ref{sec:greedyMarching}.

\section{OWNS equations}\label{sec:OWNS}

For an arbitrary system of linear first-order hyperbolic equations,
Section~\ref{sec:global} develops the global equations we wish to approximate
using one-way equations. Section~\ref{sec:Briggs} presents Briggs'
criterion~\cite{Briggs_1964_Electron}, which we use to identify right- and
left-going modes, while Section~\ref{sec:exact} uses this criterion to develop
the exact one-way projection equations. Constructing the projection operator
introduces numerical error while also entailing a high computational cost, so
OWNS-P~\cite{Towne_2022_OWNS-P} and OWNS-R~\cite{Zhu_2021_OWNS-R} approximate
its action using a rapidly-convergent \textit{recursive filter}. We conclude by
reviewing OWNS-P and OWNS-R in Sections~\ref{subsec:OWNSP}
and~\ref{subsec:OWNSR}, respectively, referring to Towne et
al.~\cite{Towne_2022_OWNS-P} and Zhu and Towne~\cite{Zhu_2021_OWNS-R} for
details on implementation. Throughout this paper, $\|\bm{b}\|$ for
$\bm{b}\in\mathbb{R}^N$ denotes the Euclidean vector norm and $\|A\|$ for
$A\in\mathbb{R}^{N\times N}$ denotes its induced matrix norm.

\subsection{Global equations}\label{sec:global}
Consider a system of first-order hyperbolic equations
\begin{subequations}
    \begin{align}
    \frac{\partial q'}{\partial t}
    +\mathcal{A}\frac{\partial q'}{\partial x}
    +\sum_{j=2}^{d}\mathcal{B}_j\frac{\partial q'}{\partial y_j}+C q'
    =f\quad&\mathrm{on}\quad\Omega\times I\label{eq:HyperbolicPDE}\\
    q'=q_0'\quad&\mathrm{on}\quad\Omega\times\{t=0\}\\
    q'=q_{\mathrm{BC}}'\quad&\mathrm{on}\quad\partial\Omega\times I,
    \end{align}
    \label{eq:HyperbolicSystem}
\end{subequations}
in $d$ spatial-dimensions over the domain $\Omega\subset\mathbb{R}^d$ and time
interval $I\equiv (0,T]$ for $T>0$, where $\mathcal{A},\mathcal{B}_{j},
\mathcal{C}\in\mathbb{R}^{d\times d}$ for $j=2,\dots,d$, and $q':\Omega\times
I\to\mathbb{R}^d$ is the solution for
initial condition $q_0':\Omega\to\mathbb{R}^d$,
boundary condition $q_{\mathrm{BC}}':\partial\Omega\times I\to\mathbb{R}^d$,
and forcing $f:\Omega\times I\to\mathbb{R}^d$. If~\eqref{eq:HyperbolicSystem}
is hyperbolic, then for all $\gamma_1,\dots,\gamma_d\in\mathbb{R}$ the matrix
$\gamma_1\mathcal{A}+\gamma_2\mathcal{B}_2+\cdots+\gamma_d\mathcal{B}_d$ has
only real eigenvalues and is diagonalizable.
Taking $\gamma_2=\cdots=\gamma_d=0$, we see that $\mathcal{A}$ is diagonalizable
with real eigenvalues
$\tilde{\mathcal{A}}=\mathcal{T}\mathcal{A}\mathcal{T}^{-1}$ and eigenvectors
$\mathcal{T}^{-1}$, which use to define the characteristic variables
$\phi=\mathcal{T}q'$. Applying finite differences with appropriate boundary
conditions in $(y_2,\dots,y_d)$ to~\eqref{eq:HyperbolicSystem} and transforming
to characteristic variables yields
\begin{subequations}
\begin{align}
    \frac{\partial \bm{\phi}'}{\partial t}
    +\tilde{A}\frac{\partial \bm{\phi}'}{\partial x}
    +\sum_{j=2}^{d}\tilde{B}_j\mathcal{D}_j\bm{\phi}'
    +\tilde{C} \bm{\phi}'=\bm{f}_{\phi}\quad&\mathrm{on}\quad[x_L,x_R]\times I\\
    \bm{q}'=\bm{q}_0'\quad&\mathrm{on}\quad[x_L,x_R]\times\{t=0\}\\
    \bm{q}'=\bm{q}_L\quad&\mathrm{on}\quad \{x=x_L\}\times I\\
    \bm{q}'=\bm{q}_R\quad&\mathrm{on}\quad \{x=x_R\}\times I
\end{align}
\label{eq:HyperbolicPDE_D_char}
\end{subequations}
for the difference operator $\mathcal{D}_j$ where $L$ and $R$ denote values
associated with the left and right boundaries, respectively. The
matrices $A,B_j,C\in\mathbb{R}^{N\times N}$ and the vectors
$\bm{q}',\bm{f}\in\mathbb{R}^N$ are the $y_j$-discretized versions of
$\mathcal{A},\mathcal{B}_j,\mathcal{C}\in\mathbb{R}^{d\times d}$ and
$q',f\in\mathbb{R}^d$, respectively. In turn this yields the matrices
$\tilde{A}=TAT^{-1}$, $\tilde{B}_j=TB_jT^{-1}$,
$\tilde{C}=TCT^{-1}+TA\partial_x T^{-1}$ and vectors
$\bm{f}_{\phi}=T\bm{f}$, $\bm{\phi}=T\bm{q}'$ in characteristic variables. We
assume that $\tilde{A}$ is non-singular (the singular case is treated by Towne
et al.\cite{Towne_2022_OWNS-P}) so that $\tilde{A}=\diag(A_{++},A_{--})$, where
$A_{++}\in\mathbb{R}^{N_+\times N_+}$ and $A_{--}\in\mathbb{R}^{N_-\times N_-}$
are positive and negative diagonal matrices, respectively. We take a Laplace
transform in time with Laplace variable $s=\eta+i\omega$ to obtain
\begin{equation}
\frac{d \hat{\bm{\phi}}}{d x} = M(s)\hat{\bm{\phi}} + \hat{\bm{g}},
\quad x \in [x_L, x_R],
\quad \hat{\bm{\phi}}(x_L) = \hat{\bm{\phi}}_{L}, 
\quad \hat{\bm{\phi}}(x_R) = \hat{\bm{\phi}}_{R}.
    \label{eq:HyperbolicPDE_D_char_freq}
\end{equation}
where $\hat{\bm{\phi}}$ and $\hat{\bm{g}}$ represent $\bm{\phi}$
and $\tilde{A}^{-1}\bm{f}_{\phi}$, respectively, in the frequency domain, while
$M(s) = -\tilde{A}^{-1} \big(s I + \sum_{j=2}^{d} \mathcal{D}_j \tilde{B}_j
+ \tilde{C} \big)$. We will ultimately set $\eta=0$ to study the
long-time stationary solution at a specified forcing frequency $\omega$. For the
moment, however, we retain $s$ instead of $i \omega$, allowing $\eta\neq0$,
which allows us to distinguish between left- and right-going eigenvalues. Since
we are only interested in long-time behavior, we have not specified initial
conditions in~\eqref{eq:HyperbolicPDE_D_char_freq}. We assume that $M(s)$ is
diagonalizable with eigenpairs $\{i\alpha_k(s),\bm{v}_k(s)\}_{k=1}^N$ since this
is true for the problems we consider, while also simplifying the exposition (the
defective case is handled by Towne and Colonius~\cite{Towne_2015_OWNS-O}).

\subsection{Briggs' criterion}\label{sec:Briggs}

Briggs' criterion (Definition~\ref{def:Briggs}) classifies eigenpairs as either
left- or right-going, while Proposition~\ref{prop:Briggs} shows that $M(s)$ has
precisely $N_+$ and $N_-$ right- and left-going eigenvalues, respectively.

\begin{definition}[Briggs' criterion~\cite{Briggs_1964_Electron}]
    Consider a wave (eigenvector) with complex wave-number (eigenvalue)
    $i\alpha(s)$. For fixed $\omega$, this wave is right-going if
    $\lim_{\eta\to\infty}\allowbreak\mathcal{I}[\alpha(s)]\to\infty$, while it
    is left-going if
    $\lim_{\eta\to\infty}\allowbreak\mathcal{I}[\alpha(s)]\to-\infty$.
    (Recall that $s=i\omega+\eta$.)
    \label{def:Briggs}
\end{definition}

\begin{proposition}\label{prop:Briggs}
For $\mathcal{R}(s)>0$, the matrix $M(s)$ has precisely $N_{+}$ and $N_{-}$
right- and left-going eigenvalues, respectively, according on Briggs'
criterion~\cite{Towne_2015_OWNS-O}.
\end{proposition}
\begin{proof}
    The eigenvalues are the solutions to the characteristic equation
    \[
    \det(\tilde{A}^{-1})^{-1}\det(M-i\alpha I)
    =\det(-sI-\sum_{j=2}^{d}\mathcal{D}_j \tilde{B}_j-i\alpha
    \tilde{A}-\tilde{C}).
    \]
    The eigenvalues are continuous functions of $s$ so taking the limit
    $\eta\to\infty$ yields $N_{+}$ eigenvalues with
    $\mathcal{I}[\alpha(s)]\to\infty$ and $N_{-}$ eigenvalues with
    $\mathcal{I}[\alpha(s)]\to-\infty$.
\end{proof}

For the rest of this paper, we take $\eta=0$, as we only need $\eta\neq0$ to
apply Briggs' criterion. We partition $M$ into left- and right-going blocks as
\begin{equation}
    M
    =
    \begin{bmatrix}
        V_+ & V_-
    \end{bmatrix}
    \begin{bmatrix}
        D_{++} & 0\\
        0 & D_{--}
    \end{bmatrix}
    \begin{bmatrix}
        V_+ & V_-
    \end{bmatrix}^{-1},
\end{equation}
where the columns of $V_+\in\mathbb{C}^{N\times N_+}$ and
$V_-\in\mathbb{C}^{N\times N_-}$ the eigenvectors associated with the right- and
left-going eigenvalues of $M$, and we have dropped the argument $s$ for brevity.
Next we introduce the coefficients $\hat{\bm{\psi}}_+\in\mathbb{C}^{N_+}$ and
$\hat{\bm{\psi}}_-\in\mathbb{C}^{N_-}$ such that
$\hat{\bm{\phi}} = V_+\hat{\bm{\psi}}_+ + V_-\hat{\bm{\psi}}_-$, and we note
that for any eigenvector $\bm{v}_k$, we have $\bm{v}_k=V\hat{\bm{\psi}}^{(k)}$
with $\hat{\psi}_l^{(k)}=\delta_{lk}$ for $k,l=1,\dots,N$ , where $\delta_{lk}$
is the Kronecker delta. For notational convenience, we define the sets
$i^{(+)}=\{1,\dots,N_+\}$ and $i^{(-)}=\{N_++1,\dots,N\}$ to denote indices
associated with right- and left-going modes, respectively. The eigenvalues need
not be unique, so we introduce $\tilde{N}_+\leq N_+$ and $\tilde{N}_-\leq N_-$
to denote the number of unique right- and left-going eigenvalues, respectively.
We additionally introduce the subsets
$\tilde{i}^{(+)}=\{ \tilde{i}^{(+)}(1),\dots,
\tilde{i}^{(+)}(\tilde{N}_+)\}\subset i^{(+)}$ and
$\tilde{i}^{(-)}=\{ \tilde{i}^{(-)}(1),\dots,
\tilde{i}^{(-)}(\tilde{N}_-)\}\subset i^{(-)}$ which denote the indices of the
unique right- and left-going eigenvalues.

\subsection{Exact projection equations}\label{sec:exact}

This section presents the one-way approximations
to~\eqref{eq:HyperbolicPDE_D_char_freq} developed in Towne et
al.~\cite{Towne_2022_OWNS-P}, which enables it to be solved as a spatial initial
value problem in $x$. Definition~\ref{def:proj_mat} introduces a projection
matrix $P$, used in Definition~\ref{def:proj_eq} to define the one-way equations
for left- and right-going modes. Proposition~\ref{prop:projection-consistent-2}
(Towne et al.~\cite{Towne_2022_OWNS-P}) shows
that~\eqref{eq:HyperbolicPDE_D_char_freq} can be solved exactly using the two
one-way equations~\eqref{eq:lowns} if and only if $\partial_x P = 0$.
Nevertheless, as discussed in Towne et al., these one-way equations can
accurately model disturbance evolution in spatially-developing flows (e.g.,
boundary layers and jets) even when $\partial_x P \neq 0$, because the flow
varies slowly in $x$ and predominantly evolves in a single direction (i.e.,
either the left- or right-going mode dominates).

\begin{definition}[Towne et al.~\cite{Towne_2022_OWNS-P}]\label{def:proj_mat}
The projection matrix
\begin{equation}
P = V E V^{-1},\quad E=\begin{bmatrix}
    I_{++} & 0\\
    0 & 0
\end{bmatrix}
\end{equation}
partitions $\hat{\bm{\phi}}$ into right- and left-going components as
$\hat{\bm{\phi}}'=P\hat{\bm{\phi}}=V_+\psi_+$ and
$\hat{\bm{\phi}}''=[I-P]\hat{\bm{\phi}}=V_-\hat{\bm{\psi}}$, respectively, where
$I_{++}\in\mathbb{R}^{N_+\times N_+}$ is the identity.
\end{definition}
\begin{proposition}[Towne et al.~\cite{Towne_2022_OWNS-P}]
    $P$ is a projection matrix.
\end{proposition}
\begin{proof}
    $E^2=E$ so that $P^2=VEV^{-1}VEV^{-1}=VEV^{-1}=P$.
\end{proof}
\begin{proposition}[Towne et al.~\cite{Towne_2022_OWNS-P}]
    $P$ commutes with $M$.
    \label{prop:commutativity}
\end{proposition}
\begin{proof}
    $ED=DE$ so that $PM=VEDV^{-1}=VDEV^{-1}=MP$.
\end{proof}

\begin{definition}[Exact one-way projection equations; Towne et
    al.~\cite{Towne_2022_OWNS-P}]\label{def:proj_eq}
The projection matrix $P$ yields the one-way equations for right- and left-going
modes:
\begin{subequations}
\begin{align}
\frac{\partial \hat{\bm{\phi}}'}{\partial x}
= P[M\hat{\bm{\phi}}' + \hat{\bm{g}}],
\quad x \in [x_L,x_R],
\quad \hat{\bm{\phi}}'(x_L) = \hat{\bm{\phi}}_L', \label{eq:lowns1} \\
\frac{\partial \hat{\bm{\phi}}''}{\partial x}
= (I-P)[M\hat{\bm{\phi}}'' + \hat{\bm{g}}],
\quad x \in [x_L,x_R],
\quad \hat{\bm{\phi}}''(x_R) = \hat{\bm{\phi}}_R''. \label{eq:lowns2}
\end{align}
\label{eq:lowns}
\end{subequations}
\end{definition}

\begin{proposition}[Towne et al.~\cite{Towne_2022_OWNS-P}] The exact one-way
    projection equations are well-posed as spatial initial value problems
    according to the criterion of Kreiss~\cite{Kreiss_1970_IBVP}.
\end{proposition}
\begin{proof}
    See the discussion by Towne and Colonius~\cite{Towne_2015_OWNS-O} and Towne
    et al.~\cite{Towne_2022_OWNS-P}, the criterion developed by
    Kreiss~\cite{Kreiss_1970_IBVP}, and the review of this criterion by
    Higdon~\cite{Higdon_1986_Kreiss}.
\end{proof}

\begin{proposition}[Towne et al.~\cite{Towne_2022_OWNS-P}]
    \label{prop:projection-consistent-2}
    If $\hat{\bm{\phi}}$ satisfies~\eqref{eq:HyperbolicPDE_D_char_freq} with
    $\hat{\bm{\phi}}(x_L)=\hat{\bm{\phi}}_L$ and
    $\hat{\bm{\phi}}(x_R)=\hat{\bm{\phi}}_R$,    
    then $P\hat{\bm{\phi}}$ and $[I-P]\hat{\bm{\phi}}$ satisfy~\eqref{eq:lowns1}
    and~\eqref{eq:lowns2} with $\hat{\bm{\phi}}'(x_L)=P\hat{\bm{\phi}}_L$ and
    $\hat{\bm{\phi}}''(x_R)=[I-P]\hat{\bm{\phi}}_R$ if and only if
    $\partial P/\partial x=0$~\cite{Towne_2022_OWNS-P}.
\end{proposition}
\begin{proof}
    See Towne et al.~\cite{Towne_2022_OWNS-P}.
\end{proof}

\subsection{Approximate projection using OWNS-P}~\label{subsec:OWNSP}
Definition~\ref{def:filter-p} introduces the OWNS-P filter, while
Proposition~\ref{prop:proj} recasts it in a matrix form, which is shown to be a
projection in Proposition~\ref{prop:approxProjectProjectMat}.
Proposition~\ref{prop:approxProjectConvergence} provides a criterion for filter
convergence, while Proposition~\ref{prop:minimalOWNS-P} provides a necessary and
sufficient condition for convergent recursion parameters to exist. We bound the
error introduced by this approximation in Proposition~\ref{prop:errOWNSP}, while
Proposition~\ref{prop:commMatOWNS-P} shows that unlike the exact projection
matrix, the approximate projection matrix does not commute with $M$, unless the
approximation is fully converged.

\begin{definition}[OWNS-P filter; Towne et al.~\cite{Towne_2022_OWNS-P}]
    \label{def:filter-p}
    Given the solution $\hat{\bm{\phi}}$, OWNS-P returns the filtered solution
    $\hat{\bm{\phi}}^0$ by solving
    \begin{subequations}
    \begin{align}
        \hat{\bm{\phi}}_+^{-N_\beta}&=0 \label{eq:filterEndPlus}\\
        (M-i\beta_-^j I)\hat{\bm{\phi}}^{-j}-(M-i\beta_+^j I)
        \hat{\bm{\phi}}^{-j-1}&=0,\quad j=1,\dots,N_\beta-1,
        \label{eq:filterTop}\\
        (M-i\beta_-^0 I)\hat{\bm{\phi}}^{0}-(M-i\beta_+^0 I)
        \hat{\bm{\phi}}^{-1}&=(M-i\beta_-^0I)\hat{\bm{\phi}},
        \label{eq:filterMiddle}\\
        (M-i\beta_+^j I)\hat{\bm{\phi}}^{j}-(M-i\beta_-^j I)
        \hat{\bm{\phi}}^{j+1}&=0,\quad j=0,\dots,N_\beta-1,
        \label{eq:filterBottom}\\
        \hat{\bm{\phi}}_-^{N_\beta}&=0 \label{eq:filterEndMinus},
    \end{align}
    \label{eq:filter-p}
    \end{subequations}
    using the recursion parameters $\{\beta_+^j,\beta_-^j\}_{j=0}^{N_\beta-1}$
    and auxiliary variables $\{\hat{\bm{\phi}}^j\}_{j=-N_\beta}^{N_\beta}$.
\end{definition}
\begin{proposition}[Towne et al.~\cite{Towne_2022_OWNS-P}]\label{prop:proj} The
    recursive filter~\eqref{eq:filter-p} can be recast in matrix form
\begin{subequations}
    \begin{equation}
    P_{N_\beta}
    =
    V R_{N_\beta}E R_{N_\beta}^{-1}V^{-1},
    \label{eq:approxProject}
    \end{equation}
    for
    \begin{equation}
    R_{N_\beta}
    =
    \begin{bmatrix}
        I_{++} & F_{++}V_{++}^{-1}V_{+-}F_{--}^{-1}\\
        F_{--}^{-1}V_{--}^{-1}V_{-+}F_{++} & I_{--}
    \end{bmatrix}^{-1},
    \label{eq:approxProjectRMat}
    \end{equation}
    where $F$ is the diagonal matrix
    \begin{equation}
        F
        =\begin{bmatrix}
            F_{++} & 0\\
            0 & F_{--}
        \end{bmatrix},\quad
        F_k=\prod_{j=0}^{N_\beta-1}
        \frac{\alpha_k - \beta_+^j}{\alpha_k- \beta_-^{j}},\quad k=1,\dots,N.
        \label{def:matF}
    \end{equation}
    \end{subequations}
\end{proposition}
\begin{proof}
    See Towne et al.~\cite{Towne_2022_OWNS-P}.
\end{proof}

\begin{proposition}[Towne~\cite{Towne_2016_Thesis}]
    \label{prop:approxProjectProjectMat}
$P_{N_\beta}$ is a projection matrix.
\end{proposition}
\begin{proof}
$P_{N_\beta}^2
=VR_{N_\beta}E R_{N_\beta}^{-1}V^{-1}VR_{N_\beta}ER_{N_\beta}^{-1}V^{-1}
=P_{N_\beta}$.
\end{proof}

\begin{proposition}[Towne et al.~\cite{Towne_2022_OWNS-P}]
    \label{prop:approxProjectConvergence} $P_{N_\beta}\to P$ if and only if
    \begin{equation}
    \lim_{N_\beta\to\infty}
    \prod_{j=0}^{N_\beta-1}\frac{|\alpha_m-\beta_{+}^{j}|}
    {|\alpha_m-\beta_-^j|}
    \frac{|\alpha_n-\beta_{-}^{j}|}
    {|\alpha_n-\beta_+^j|}=0,\quad\forall(m,n)\in i^{(+)}\times i^{(-)}.
    \label{eq:fConv}
    \end{equation}
\end{proposition}
\begin{proof}
    See Towne et al.~\cite{Towne_2022_OWNS-P}.
\end{proof}

\begin{proposition}[Towne et al.~\cite{Towne_2022_OWNS-P}]
    \label{prop:minimalOWNS-P}
\begin{subequations}
    Recursion parameters such that $P_{N_\beta}\to P$ exist if and only if
    $\alpha_m\neq\alpha_n$ for all $(m,n)\in i^{(+)}\times i^{(-)}$. If
    $\tilde{N}_+\leq \tilde{N}_-$, then take $N_\beta=\tilde{N}_+$ with
    \begin{equation}
        \beta_+^{j-1}=\alpha_{\tilde{i}^{(+)}(j)},\quad
        \beta_-^{j-1}\neq\alpha_{\tilde{i}^{(+)}(j)},\quad
        j=1,\dots,\tilde{N}_+,
        \label{eq:owns-p-params_plus}
    \end{equation}
    while if $\tilde{N}_-<\tilde{N}_+$, then take $N_\beta=\tilde{N}_-$ with
    \begin{equation}
        \beta_-^{j-1}=\alpha_{\tilde{i}^{(-)}(j)},\quad
        \beta_+^{j-1}\neq\alpha_{\tilde{i}^{(-)}(j)},\quad
        j=1,\dots,\tilde{N}_-.
        \label{eq:owns-p-params_minus}
    \end{equation}
\end{subequations}
\end{proposition}
\begin{proof}
    See Towne et al.~\cite{Towne_2022_OWNS-P} or Appendix~\ref{app:Proofs}.
\end{proof}

Proposition~\ref{prop:errOWNSP} bounds the error in the OWNS-P approximation,
$\|P-P_{N_\beta}\|$, for small $\|F_{++}\|\|F_{--}^{-1}\|$. We use this error
bound to define our objective function for the subset selection
problem~\eqref{eq:subsetSelection}.

\begin{proposition}\label{prop:errOWNSP}
    If $\|F_{++}\|\|F_{--}^{-1}\|<\epsilon$ where
    \begin{equation}
        \epsilon \equiv \min\{\hat{\epsilon},
        \|V_{++}^{-1}V_{+-}\|^{-1/2}\|V_{--}^{-1}V_{-+}\|^{-1/2}\}
    \label{eq:errBoundEpsilon}
    \end{equation}
    for small $\hat{\epsilon}>0$, then the OWNS-P error is bounded as
    \begin{equation}
        \|P_{N_\beta}-P\|\leq \|V\|\|F_{++}\|\|F_{--}^{-1}\|\big(
            \|V_{++}^{-1}V_{+-}\|+\|V_{--}^{-1}V_{-+}\|\big)\|V^{-1}\|
            +\mathcal{O}(\epsilon^2).
        \label{eq:errOWNSP}
    \end{equation}
\end{proposition}
\begin{proof}
    See Appendix~\ref{app:Proofs}.
\end{proof}

\begin{proposition}~\label{prop:greedy-p}
    If $\alpha_m\in\{\beta_+^j\}_{j=0}^{N_\beta-1}$ for any $m\in i^{(+)}$, then
    $P_{N_\beta}\bm{v}_m=\bm{v}_m$; if
    $\alpha_n\in\{\beta_-^j\}_{j=0}^{N_\beta-1}$ for any $n\in i^{(-)}$,
    then $P_{N_\beta}\bm{v}_n=0$.
\end{proposition}

\begin{proof}
    Since $(M-i\beta I)\hat{\bm{\phi}}=V(D-i\beta I)\bm{\psi}$ where $V$ has
    full rank and $(D-i\beta I)$, we diagonalize~\eqref{eq:filterTop}
    through~\eqref{eq:filterBottom} as
    \begin{subequations}
    \begin{align}
        (\alpha_k-\beta_-^j)\hat{\psi}_{k}^{-j}-(\alpha_k-\beta_+^j)
        \hat{\psi}_{k}^{-j-1}&=0,\quad j=1,\dots,N_\beta-1,
        \label{eq:OWNS-P-diag-minus}\\
        (\alpha_k-\beta_-^0)\hat{\psi}_{k}^{0}-(\alpha_k-\beta_+^0)
        \hat{\psi}_{k}^{-1}&=(\alpha_k-\beta_-^0)\hat{\psi}_k,
        \label{eq:OWNS-P-diag-zero}\\
        (\alpha_k-\beta_+^j)\hat{\psi}_{k}^{j}-(\alpha_k-\beta_-^j)
        \hat{\psi}_{k}^{j+1}&=0,\quad j=0,\dots,N_\beta-1
        \label{eq:OWNS-P-diag-plus},
    \end{align}
    \end{subequations}
    for $k=1,\dots,N$. Using~\eqref{eq:OWNS-P-diag-minus}
    and~\eqref{eq:OWNS-P-diag-zero}, we see that
    \[
    \hat{\psi}_m^{0}-\hat{\psi}_m=\prod_{j=0}^{N_\beta-1}
    \frac{\alpha_m - \beta_+^j}{\alpha_m- \beta_-^{j}}\hat{\psi}_m^{-N_\beta}=0,
    \]
    since $\alpha_m\in\{\beta_+^j\}_{j=0}^{N_\beta-1}$. For $\bm{v}_m$, recall
    that $\hat{\psi}_m^{(l)}=\delta_{lm}$, so that
    $\hat{\psi}_m^{(m)0}=\hat{\psi}_m^{(m)}$ implies
    $P_{N_\beta}\bm{v}_m=\bm{v}_m$. Using~\eqref{eq:OWNS-P-diag-zero}
    and~\eqref{eq:OWNS-P-diag-plus}, we see that
    \[
    \hat{\psi}_n^0=\prod_{j=0}^{N_\beta-1}
    \frac{\alpha_n - \beta_-^j}{\alpha_n- \beta_+^{j}}\hat{\psi}_n^{ N_\beta}=0
    \]
    since $\alpha_n\in\{\beta_-^j\}_{j=0}^{N_\beta-1}$, so that for $\bm{v}_n$,
    $\hat{\psi}_n^{(n)0}=0$ implies $P_{N_\beta}\bm{v}_n=0$ since
    $\hat{\psi}_n^{(l)}=\delta_{ln}$.
\end{proof}

\begin{remark}
Proposition~\ref{prop:errOWNSP} shows that the error $\|P_{N_\beta}-P\|$ scales
with
\penalty 200
$\|F_{++}\|\|F_{--}\|$ for small $\|F_{++}\|\|F_{--}\|$, while
Proposition~\ref{prop:greedy-p} shows it is always possible to achieve zero
error for any eigenvector by using its eigenvalue as a recursion parameter.
\end{remark}

\begin{proposition}\label{prop:commMatOWNS-P}
    $P_{N_\beta}$ commutes with $M$ if and only if
    $\|F_{++}V_{++}^{-1}V_{+-}F_{--}^{-1}\|\to0$ and
    $\|F_{--}^{-1}V_{--}^{-1}V_{-+}F_{++}\|\to0$.
\end{proposition}
\begin{proof}
    If $\|F_{++}V_{++}^{-1}V_{+-}F_{--}^{-1}\|\to0$ and
    $\|F_{--}^{-1}V_{--}^{-1}V_{-+}F_{++}\|\to0$, then $R_{N_\beta}^{-1}\to I$
    and $P_{N_\beta}M=PM=MP=MP_{N_\beta}$. Otherwise,
    $R_{N_\beta}^{-1}D\neq D R_{N_\beta}$, and $P_{N_\beta}$ does not commute
    with $M$.
\end{proof}

\begin{remark}
Proposition~\ref{prop:approxProjectProjectMat} establishes that $P_{N_\beta}$ is
always a projection matrix, while Proposition~\ref{prop:commMatOWNS-P}, in
conjunction with Proposition~\ref{prop:approxProjectConvergence}, shows that
although $P_{N_\beta}$ approximates $P$, it only commutes with $M$ if the
approximation is fully converged. Towne~\cite{Towne_2016_Thesis} proves that
OWNS-P yields a projection matrix, but does not discuss whether it commutes with
$M$. Similarly, Zhu and Towne~\cite{Zhu_2021_OWNS-R} demonstrate that OWNS-P
approximates the eigenvectors in $P$, while OWNS-R approximates the eigenvalues,
yet they do not address whether these approximations commute with $M$ or form
projection operators.
\end{remark}

\subsection{Approximate projection using OWNS-R}\label{subsec:OWNSR}

OWNS-R constructs approximations to $P$ with a reduced computational cost
relative to OWNS-P. Definition~\ref{def:filter-r} presents the OWNS-R recursive
filter, which is recast in matrix form in
Proposition~\ref{prop:owns-r-implementation}. In contrast to OWNS-P,
Proposition~\ref{prop:projMatOWNS-R} shows that the resulting matrix is not
generally a projection matrix, while Proposition~\ref{prop:commMatOWNS-R} shows
that it always commutes with $M$. Proposition~\ref{prop:OWNS-R-converge}
provides a criterion for convergence, while
Proposition~\ref{prop:minimalOWNS-R} provides a necessary and sufficient
condition for convergent parameters to exist. While Zhu and
Towne~\cite{Zhu_2021_OWNS-R} suggested that OWNS-R convergence can be guaranteed
using the same recursion parameters as OWNS-P, we show in
Proposition~\ref{prop:minimalOWNS-R-false} that this is not true. In addition,
Proposition~\ref{prop:owns-r-repeat-blowup} shows that repeated applications of
the OWNS-R matrix leads to unbounded growth or decay unless the filter is fully
converged.

\begin{definition}[OWNS-R filter; Zhu and Towne~\cite{Zhu_2021_OWNS-R}]
    \label{def:filter-r}
    Given the solution $\hat{\bm{\phi}}$, OWNS-R returns the filtered solution
    $\hat{\bm{\phi}}^{N_\beta}$ by solving
    \begin{subequations}
        \begin{align}
            \hat{\bm{\phi}}^0&=\frac{1}{h}\hat{\bm{\phi}},\\
            (M-i\beta_*^j)\hat{\bm{\phi}}^j
            &=
            (M-i\beta_-^j)\hat{\bm{\phi}}^{j-1},\quad j=1,\dots,N_\beta,
        \end{align}
        \label{eq:filter-r}
    \end{subequations}
         where $\{\beta_+^j,\beta_-^j\}_{j=1}^{N_\beta}$ are the recursion
         parameters and $\{\hat{\bm{\phi}}^j\}_{j=0}^{N_\beta}$ are the
         auxiliary variables, while $\{\beta^j_*\}_{j=1}^{N_\beta}$ and $h$
         satisfy
        \begin{equation}
            h\prod_{j=1}^{N_\beta}(\alpha-\beta^j_*)
            =\prod_{j=1}^{N_\beta}(\alpha-\beta^j_-)
            +c\prod_{j=1}^{N_\beta}(\alpha-\beta^j_+),
            \label{eq:owns-r-polynomial}
        \end{equation}
    for freely-chosen $c\geq0$ (see Remark~\ref{rmk:owns-r-c} for more
    information on choosing $c$). Note that $\{\beta^j_*\}_{j=1}^{N_\beta}$ and
    $h$ must be computed numerically (e.g., via \texttt{roots} in MATLAB).
\end{definition}

\begin{proposition}[Zhu and Towne~\cite{Zhu_2021_OWNS-R}]
    \label{prop:owns-r-implementation}
    The recursive filter~\eqref{eq:filter-r} can be recast in matrix form
    \begin{equation}
        P^{(R)}_{N_{\beta}}
        =V E_{N_\beta} V^{-1},\quad E_{N_\beta}=(1+cF)^{-1},
        \label{eq:matOWNSR}
    \end{equation}
    where
    \begin{equation}
        E_{N_\beta}^{(k)}
        =
        \frac{\prod_{j=1}^{N_\beta}(\alpha_k-\beta_-^j)}
        {\prod_{j=1}^{N_\beta}(\alpha_k-\beta_-^j) +
         c \prod_{j=1}^{N_\beta}(\alpha_k-\beta_+^j)},\quad k=1,\dots,N,
    \end{equation}
    for $F$ defined in~\eqref{def:matF}.
\end{proposition}
\begin{proof}
    See Zhu and Towne~\cite{Zhu_2021_OWNS-R}.
\end{proof}

\begin{remark}\label{rmk:owns-r-c}
    Zhu and Towne~\cite{Zhu_2021_OWNS-R} recommend setting $c\approx1$ since
    $c=0$ retains all modes, while $c\to\infty$ removes all modes. Therefore, we
    set $c=1$.
\end{remark}

\begin{proposition}\label{prop:projMatOWNS-R}
    $P_{N_\beta}^{(R)}$ is a projection matrix if and only if
    $E_{N_\beta}^{(k)}\in\{0,1\}$ for all $k=1,\dots,N$.
\end{proposition}
\begin{proof}
    We have $(P^{(R)}_{N_{\beta}})^2=VE^2_{N_\beta}V^{-1}$, so that
    $P_{N_\beta}^{(R)}$ is a projection matrix if and only if
    $E_{N_\beta,k}^2=E_{N_\beta,k}$ for $k=1,\dots,N$, which requires
    $E_{N_\beta,k}\in\{0,1\}$.
\end{proof}

\begin{proposition}\label{prop:commMatOWNS-R}
    $P_{N_\beta}^{(R)}$ commutes with $M$.
\end{proposition}
\begin{proof}
$P^{(R)}_{N_\beta}M=VE_{N_\beta}DV^{-1}=VDE_{N_\beta}V^{-1}=MP^{(R)}_{N_\beta}$
since $E_{N_\beta}$ and $D$ are diagonal matrices.
\end{proof}

\begin{proposition}[Zhu and Towne~\cite{Zhu_2021_OWNS-R}]
    \label{prop:OWNS-R-converge}
$P_{N_\beta}^{(R)}\to P$ if and only if $E_{N_\beta}^{(m)}\to 1$ for all
$m\in i^{(+)}$ and $E_{N_\beta}^{(n)}\to0$ for all $n\in i^{(-)}$.
\end{proposition}
\begin{proof}
    See Zhu and Towne~\cite{Zhu_2021_OWNS-R}.
\end{proof}

\begin{proposition}\label{prop:minimalOWNS-R}
    Recursion parameters such that $P_{N_\beta}^{(R)}\to P$ exist if and only if
    $\alpha_m\neq\alpha_n$ for all $(m,n)\in i^{(+)}\times i^{(-)}$. In
    particular, take $N_\beta=\tilde{N}_++\tilde{N}_-$ with
    \begin{subequations}
    \begin{align}
        \beta_+^{j}&=\alpha_{\tilde{i}^{(+)}(j)},\quad
        \beta_-^{j}\neq\alpha_{\tilde{i}^{(+)}(j)},\quad
        j=1,\dots,\tilde{N}_+,\\
        \beta_-^{\tilde{N}_++j}&=\alpha_{\tilde{i}^{(-)}(j)},\quad
        \beta_+^{\tilde{N}_++j}\neq\alpha_{\tilde{i}^{(-)}(j)},\quad
        j=1,\dots,\tilde{N}_-.
    \end{align}
    \label{eq:OWNS-R-params}
    \end{subequations}
\end{proposition}
\begin{proof}
    If there exists $(\hat{m},\hat{n})\in i^{(+)}\times i^{(-)}$ such that
    $\alpha_{\hat{m}}=\alpha_{\hat{n}}$, then
    $E_{N_\beta}^{(\hat{m})}=E_{N_\beta}^{(\hat{n})}$, but convergence requires
    $E_{N_\beta}^{(\hat{m})}\to1$ and $E_{N_\beta}^{(\hat{n})}\to0$, which is
    not possible. If $\alpha_m\neq\alpha_n$ for
    $(m,n)\in i^{(+)}\times i^{(-)}$, then choose recursion parameters according
    to~\eqref{eq:OWNS-R-params} so that
    \begin{align*}
    E_{\tilde{N}_++\tilde{N}_-}^{(m)}
    &=\frac{\prod_{j=1}^{\tilde{N}_++\tilde{N}_-}(\alpha_m-\beta_-^j)}
    {\prod_{j=1}^{\tilde{N}_++\tilde{N}_-}(\alpha_m-\beta_-^j)+0}
    =1,\quad \forall m\in i^{(+)},\\
    E_{\tilde{N}_++\tilde{N}_-}^{(n)}
    &=\frac{0}{0+c\prod_{j=1}^{\tilde{N}_++\tilde{N}_-}(\alpha_n-\beta_+^j)}
    =0,\quad \forall n\in i^{(-)},
    \end{align*}
    where we have used that
    $\alpha_m \in \{\beta_+^j\}$, $\alpha_m \notin \{\beta_-^j\}$ for all
    $m \in i^{(+)}$ and $\alpha_n \in \{\beta_-^j\},$
    $\alpha_n \notin \{\beta_+^j\}$ for all $n \in i^{(-)}$.    
\end{proof}

\begin{proposition}\label{prop:minimalOWNS-R-false}
    The recursion parameters that guarantee OWNS-P convergence
    (Proposition~\ref{prop:minimalOWNS-P}) do not guarantee OWNS-R convergence.
\end{proposition}
\begin{proof}
    If $\tilde{N}_+\leq \tilde{N}_-$, use~\eqref{eq:owns-p-params_plus}. Then
    there exists $n\in i^{(-)}$ such that
    $\alpha_n\notin\{\beta_-^j\}_{j=1}^{N_\beta}$ so that
    $E_{N_\beta}^{(n)}\neq0$ and $P_{N_\beta}^{(R)}\neq P$. If
    $\tilde{N}_-<\tilde{N}_+$, use~\eqref{eq:owns-p-params_minus}. Then there
    exists $m\in i^{(+)}$ such that
    $\alpha_m\notin\{\beta_+^j\}_{j=1}^{N_\beta}$ so that
    $E_{N_\beta}^{(m)}\neq1$ and $P_{N_\beta}^{(R)}\neq P$.
\end{proof}

\begin{proposition}\label{prop:minimalOWNS-R-true}
    The recursion parameters that guarantee OWNS-R convergence
    (Proposition~\ref{prop:minimalOWNS-R}) also guarantee OWNS-P convergence.
\end{proposition}
\begin{proof}
Taking $N_\beta=N$ with parameters~\eqref{eq:OWNS-R-params} guarantees that
$\alpha_m\in\{\beta_+^j\}_{j=1}^{N_\beta}$ for all $m\in i^{(+)}$ and
$\alpha_n\in\{\beta_-^j\}_{j=1}^{N_\beta}$ for all $n\in i^{(-)}$, so that the
convergence criterion~\eqref{eq:fConv} is satisfied (see
Proposition~\ref{prop:approxProjectConvergence}).
\end{proof}

\begin{proposition}[Zhu and Towne~\cite{Zhu_2021_OWNS-R}]\label{prop:errOWNSR}
    The error introduced by OWNS-R is bounded by
    \begin{equation}
        \|P-P_{N_\beta}^{(R)}\|\leq
        \max\big\{ |c|\|F_{++}\|, \|F_{--}^{-1}\| \big\}\|V\|\|V^{-1}\|
        +\mathcal{O}(\epsilon^2),
    \end{equation}
    where $\|F_{++}\|,\|F_{--}\|<\epsilon$ for small $\epsilon>0$ such that
    $\epsilon\ll1$.
\end{proposition}

\begin{proof}
    See Zhu and Towne~\cite{Zhu_2021_OWNS-R}.
\end{proof}

\begin{proposition}~\label{prop:greedy-r}
    If $\alpha_m\in\{\beta_+^j\}_{j=1}^{N_\beta}$ for any $m\in i^{(+)}$, then
    $P_{N_\beta}^{(R)}\bm{v}_m=\bm{v}_m$; if
    $\alpha_n\in\{\beta_-^j\}_{j=1}^{N_\beta}$ for any $n\in i^{(-)}$, then
    $P_{N_\beta}^{(R)}\bm{v}_n=0$.
\end{proposition}

\begin{proof}
    If $\alpha_m\in\{\beta_+^j\}_{j=1}^{N_\beta}$ for $m\in i^{(+)}$, then
    $E_{N_\beta}^{(m)}=1$ so that $P_{N_\beta}^{(R)}\bm{v}_{m}=\bm{v}_m$. If
    $\alpha_n\in\{\beta_-^j\}_{j=1}^{N_\beta}$ for $n\in i^{(-)}$, then
    $E_{N_\beta}^{(n)}=0$ so that $P_{N_\beta}^{(R)}\bm{v}_{n}=0$.
\end{proof}

\begin{remark}
Proposition~\ref{prop:errOWNSR} establishes that the error in introduced by
OWNS-R scales with $\max\{ |c|\|F_{++}\|, \|F_{--}^{-1}\| \}$ for sufficiently
small $\|F_{++}\|$ and $\|F_{--}\|$, while Proposition~\ref{prop:greedy-p} shows
it is always possible to achieve zero error for any eigenvector by using its
eigenvalue as a recursion parameter.
\label{rmk:error-compare}
\end{remark}

\begin{proposition}\label{prop:owns-r-repeat-blowup}
For any $k=1,\dots,N$, repeated application of $P_{N_\beta}^{(R)}$ causes
unbounded growth of $\bm{v}_k$ when $|E_{N_\beta}^{(k)}|>1$ and decay to zero
when $|E_{N_\beta}^{(k)}|<1$.
\end{proposition}
\begin{proof}
    Applying $P_{N_\beta}^{(R)}$ $n$ times yields
    $(P_{N_\beta}^{(R)})^n\bm{v}_k = V (E_{N_\beta})^n V^{-1}\bm{v}_k$. If
    $|E_{N_\beta}^{(k)}| < 1$, then mode $k$ will be removed
    ($\lim_{n\to\infty}|E_{N_\beta}^{(k)}|=0$), which is undesirable for
    right-going modes. In contrast, it will grow without bound if
    $|E_{N_\beta}^{(k)}| > 1$ ($\lim_{n\to\infty}|E_{N_\beta}^{(k)}|=\infty$),
    which is undesirable for both right- and left-going modes.
\end{proof}

\begin{remark}
Proposition~\ref{prop:projMatOWNS-R} shows that $P_{N_\beta}^{(R)}$ is not
generally a projection matrix, while
Proposition~\ref{prop:owns-r-repeat-blowup} shows that repeated application of
$P_{N_\beta}^{(n)}$ introduces additional error unless $|E_{N_\beta}^{(n)}|<1$
for all $n\in i^{(-)}$, and $E_{N_\beta}^{(m)}=1$ for all $m\in i^{(+)}$, so
that it should only be applied once. In contrast, OWNS-P does not have this
limitation since $P_{N_\beta}$ is always a projection matrix by
Proposition~\ref{prop:approxProjectProjectMat}.
\end{remark}

\begin{remark}~\label{rmk:rounding}
    In theory, the OWNS-R error decreases with increasing $N_\beta$ if both
    $\|F_{++}\|$ and $\|F_{--}^{-1}\|$ decrease. In practice, we observe that
    for sufficiently large $N_\beta$, rounding errors due to finite precision
    arithmetic prevent the OWNS-R filter from converging. This error can
    partially be mitigated by computing $\{\beta_*^j\}_{j=1}^{N_\beta}$ to quad,
    rather than double precision~\cite{Zhu_2021_OWNS-R}, but error associated
    with applying the filter persists. We compute
    $\{\beta_*^j\}_{j=1}^{N_\beta}$ to quad precision in Section 5, while we use
    double precision in Section 6 because $N_\beta$ is small enough that quad
    precision is unnecessary.
\end{remark}

\begin{remark}
    Despite its slower convergence and rounding errors for large $N_\beta$,
    OWNS-R is advantageous because it entails a lower computational cost than
    OWNS-P and scales better for large systems of
    equations~\cite{Zhu_2021_OWNS-R}. We refer to Zhu and
    Towne~\cite{Zhu_2021_OWNS-R} for a more detailed discussion of computational
    cost of OWNS-P and OWNS-R.
\end{remark}

\subsection{Summary and comparison of the OWNS formulations}

The OWNS-P error scales with $\|F_{++}\|\|F_{--}^{-1}\|$, so that we have
convergence with either $\|F_{--}^{-1}\|=0$ or $\|F_{++}\|=0$, while the
OWNS-R error scales with $\max\{|c|\|F_{++}\|, \|F_{--}^{-1}\| \}$ so that we
must instead have $\|F_{--}^{-1}\|=\|F_{++}\|=0$ for convergence. Thus, OWNS-R
will generally require larger $N_\beta$. Although similar strategies can be used
to pick recursion parameters for both OWNS-P and OWNS-R,
Proposition~\ref{prop:minimalOWNS-R-false} shows that there exists parameter
sets for which OWNS-P is converged while OWNS-R is not.

Whereas OWNS-P yields a projection matrix that generally does not commute with
$M$, OWNS-R yields a matrix that commutes with $M$ but is not generally a
projection matrix, so that OWNS-P and OWNS-R each lose one of the properties of
$P$ (projection matrix or commutation with $M$). Although $P_{N_\beta}$ is a
projection matrix, we will have $P_{N_\beta}\hat{\bm{\phi}}_{N_\beta}'
\neq\hat{\bm{\phi}}_{N_\beta}'$, where $\hat{\bm{\phi}}'_{N_\beta}$ is the
solution to the OWNS-P approximation to~\eqref{eq:lowns1}, since
$P_{N_\beta}M=M P_{N_\beta}$ (unless the approximation is fully converged).
Although $P_{N_\beta}^{(R)}$ commutes with $M$, we will have
$P_{N_\beta}^{(R)}\hat{\bm{\phi}}_{N_\beta}^{(R)\prime}
\neq\hat{\bm{\phi}}_{N_\beta}^{(R)\prime}$, where
 $\hat{\bm{\phi}}^{(R)\prime}_{N_\beta}$ is the solution to the OWNS-R
 approximation to~\eqref{eq:lowns1}, since
 $P_{N_\beta}^{(R)}\neq P_{N_\beta}^{(R)}$ (unless the approximation is fully
 converged). Moreover, by Proposition~\ref{prop:owns-r-repeat-blowup}, repeated
 applications of $P_{N_\beta}^{(R)}$ should be avoided to prevent non-physical
 solution blow-up or decay.

\section{Greedy algorithm for recursion parameter selection}\label{sec:greedy}

If $\alpha_m \neq \alpha_n$ for all $(m,n) \in i^{(+)}\times i^{(-)}$, then
Propositions~\ref{prop:minimalOWNS-P} and~\ref{prop:minimalOWNS-R} identify
recursion parameter sets that guarantee convergence. However, these sets are
large: OWNS-P requires solving the linear systems~\eqref{eq:filter-p}, while the
OWNS-R filter~\eqref{eq:filter-r} becomes increasingly sensitive to rounding
errors as $N_\beta$ grows. In addition, previous work~\cite{Towne_2015_OWNS-O,
Towne_2022_OWNS-P,Zhu_2021_OWNS-R} has avoided computing the eigenvalues of $M$
due to its relatively high computational cost. Consequently, previous work has
relied on \textit{heuristic} estimates of key eigenvalues to construct stable
and accurate OWNS approximations with $N_\beta \ll N$. This approach has two
drawbacks: it requires problem-specific parameter tuning and typically yields
slow convergence. In contrast to previous work, we choose to accept the
computational cost of computing the eigenvalues of $M$, with the goal of using
this information to choose parameters that converge faster, thus reducing the
net cost of OWNS.

In the present work, we formulate an optimization problem that enables automatic
recursion parameter selection with rapid error convergence. Let
\begin{equation}
    \mathcal{S}_+ = \{ \alpha_+^{m} \mid m = 1,\dots,\tilde{N}_+\},\quad
    \mathcal{S}_- = \{ \alpha_-^{n} \mid n = 1,\dots,\tilde{N}_-\},
\end{equation}
denote the sets of right- and left-going eigenvalues.
Propositions~\ref{prop:minimalOWNS-P} and~\ref{prop:minimalOWNS-R} guarantee the
existence of subsets $\{\beta_+^j\}_{j=1}^{k} \subset \mathcal{S}_+$ and
$\{\beta_-^j\}_{j=1}^{k} \subset \mathcal{S}_-$
such that the OWNS approximations converge, with
$k=\min\{\tilde{N}_-,\tilde{N}_+\}$ for OWNS-P and $k=\tilde{N}_-+\tilde{N}_+$
for OWNS-R. 
For $N_\beta \ll N$, we formulate a \textit{subset selection} problem: select
$\{\beta_+^j\}_{j=1}^{N_\beta} \subset \mathcal{S}_+$ and
$\{\beta_-^j\}_{j=1}^{N_\beta} \subset \mathcal{S}_-$ to minimize the error in
the OWNS approximations. Subset selection is
NP-hard~\cite{Natarajan_1995_OptimalSubset} and cannot be solved in polynomial
time. Following Das and Kempe (2011)~\cite{Das_2011_Greedy}, we apply a greedy
algorithm to solve the subset selection problem, using the error bounds from
Propositions~\ref{prop:errOWNSP} and~\ref{prop:errOWNSR} as the objective
functions.

\subsection{Greedy subset selection}

We define the recursion parameter subsets
$\Xi_{N_\beta}^{+} \equiv \{\beta_+^j\}_{j=1}^{N_\beta} \subset \mathcal{S}_+$
and
$\Xi_{N_\beta}^{-} \equiv \{\beta_-^j\}_{j=1}^{N_\beta} \subset \mathcal{S}_-$,
and the functions
\begin{subequations}
\begin{equation}
\hat{\mathcal{J}}_+(\alpha;\Xi_{N_\beta}^{+},\Xi_{N_\beta}^{-}) \equiv
\prod_{j=1}^{N_\beta} \frac{|\alpha - \beta_+^j|}{|\alpha - \beta_-^j|}, \quad
\hat{\mathcal{J}}_-(\alpha;\Xi_{N_\beta}^{+},\Xi_{N_\beta}^{-}) \equiv
\prod_{j=1}^{N_\beta} \frac{|\alpha - \beta_-^j|}{|\alpha - \beta_+^j|}.
\end{equation}
By Proposition~\ref{prop:errOWNSP}, the OWNS-P error scales as
\begin{equation}
\mathcal{J}(\Xi_{N_\beta}^{+},\Xi_{N_\beta}^{-})
\equiv
\max_{\alpha\in\mathcal{S}_+}
\hat{\mathcal{J}}_+(\alpha;\Xi_{N_\beta}^{+},\Xi_{N_\beta}^{-})
\;\max_{\alpha\in\mathcal{S}_-}
\hat{\mathcal{J}}_-(\alpha;\Xi_{N_\beta}^{+},\Xi_{N_\beta}^{-}),
\label{eq:objFncRecursions}
\end{equation}
and by Proposition~\ref{prop:errOWNSR}, the OWNS-R error scales as
\begin{equation}
\mathcal{J}^{(R)}(\Xi_{N_\beta}^{+},\Xi_{N_\beta}^{-}) \equiv
\max\Big\{
\max_{\alpha\in\mathcal{S}_+}\hat{\mathcal{J}}_+
(\alpha;\Xi_{N_\beta}^{+},\Xi_{N_\beta}^{-}),
\;\max_{\alpha\in\mathcal{S}_-}
\hat{\mathcal{J}}_-(\alpha;\Xi_{N_\beta}^{+},\Xi_{N_\beta}^{-})
\Big\}.
\label{eq:objFncRecursionsR}
\end{equation}
\end{subequations}
For brevity, we will use $\mathcal{J}$ to refer to
both~\eqref{eq:objFncRecursions} and~\eqref{eq:objFncRecursionsR}.

Our subset selection problem becomes: for $N_\beta\ll N$, solve
\begin{equation}
    \min_{
    \Xi^+_{N_\beta}\subset\mathcal{S}_+,
    \Xi^-_{N_\beta}\subset\mathcal{S}_-
    }
    \mathcal{J}(\Xi_{N_\beta}^+,\Xi_{N_\beta}^-).
    \label{eq:subsetSelection}
\end{equation}
This problem is well posed because the objective function is continuous (it is a
quotient of polynomials with a non-zero denominator), but it is
NP-hard~\cite{Natarajan_1995_OptimalSubset}. We solve it using the greedy
algorithm outlined in Algorithm~\ref{alg:greedyRecursion}. In words, we randomly
initialize $\Xi_1^+$ and $\Xi_1^-$ with elements of $\mathcal{S}_+$ and
$\mathcal{S}_-$, respectively. At each iteration we use to
$\hat{\mathcal{J}}_+$ and $\hat{\mathcal{J}}_-$ to identify the
worst-approximated  elements of $\mathcal{S}_+$ and $\mathcal{S}_-$, which we
append to our recursion parameter set. This strategy leads to the largest local
decrease in $\mathcal{J}$.

\begin{remark}\label{rmk:tracking}
    Previous work on OWNS has combined heuristic parameter selection with
    mode-tracking, where important discrete modes (e.g., Kelvin-Helmholtz in
    free-shear flows or Mack's second mode in high-speed boundary layers) were
    added to the recursion parameter set to improve approximation
    accuracy~\cite{Towne_2015_OWNS-O,Towne_2022_OWNS-P,Zhu_2021_OWNS-R}. While
    previous methods manually identified eigenvalues to track, the present work
    proposes an automatic approach.
\end{remark}
\begin{remark}
    The greedy algorithm can be initialized with user-specified recursion
    parameters, but we have found that it is better to initialize the algorithm
    with many different randomly chosen pairs of left- and right-going
    eigenvalues, and to choose the parameter set with the smallest
    $\mathcal{J}$.
\end{remark}

\begin{algorithm}
\caption{Greedy recursion parameter selection}\label{alg:greedyRecursion}
\begin{algorithmic}
\STATE Choose $\beta_+^1 \in \mathcal{S}_+$ and $\beta_-^1 \in \mathcal{S}_-$
randomly
\STATE $\Xi_1^+ \gets \{\beta_+^1\}, \quad \Xi_1^- \gets \{\beta_-^1\}$
\STATE $N_g \gets 1$
\WHILE{$N_g < N_\beta$}
\STATE $\beta_+^{N_g+1} \gets \arg\max_{\alpha\in\mathcal{S}_+}
\hat{\mathcal{J}}_+(\alpha;\Xi_{N_g}^+,\Xi_{N_g}^-)$
\STATE $\beta_-^{N_g+1} \gets \arg\max_{\alpha\in\mathcal{S}_-}
\hat{\mathcal{J}}_-(\alpha;\Xi_{N_g}^+,\Xi_{N_g}^-)$
\STATE $\Xi_{N_g+1}^+ \gets \Xi_{N_g}^+\cup \{\beta_+^{N_g+1}\}$
\STATE $\Xi_{N_g+1}^- \gets \Xi_{N_g}^-\cup \{\beta_-^{N_g+1}\}$
\STATE $N_g \gets N_g + 1$
\ENDWHILE
\end{algorithmic}
\end{algorithm}

\subsection{Numerical stability of the recursive filter}

OWNS solves equations of the form $(M-i\beta_-^j)\hat{\bm{\phi}}^{j+1}
=(M-i\beta_+^j)\hat{\bm{\phi}}^{j}$ for $j=0,\dots,N_\beta-1$. If the exact
solution is given by $\hat{\bm{\varphi}}^{j+1}$, then the rounding error
introduced by solving the linear system is bounded by
\[
    \frac{\|\hat{\bm{\phi}}^{j+1}-\hat{\bm{\varphi}}^{j+1}\|}
    {\|\hat{\bm{\varphi}}^{j+1}\|}
    \leq
    \kappa(M-i\beta_-^jI)\frac{\|M-i\beta_-^jI\|\|\hat{\bm{\phi}}^{j+1}-
    \hat{\bm{\phi}}^{j}\|+|\beta_+^j-\beta_-^j|\|\hat{\bm{\phi}}^j\|}
    {\|(M-i\beta_+I)\hat{\bm{\varphi}}^j\|}.
\]
If $\beta_+^j=\beta_-^j$, then $\|\hat{\bm{\phi}}^{j+1}-\hat{\bm{\phi}}^{j}\|
=\mathcal{O}(\epsilon_{\mathrm{mach}})$, where $\epsilon_{\mathrm{mach}}$ is
machine epsilon, and the error is
$\mathcal{O}(\epsilon_{\mathrm{mach}}\kappa(M-i\beta_-^jI))$. However, this
error grows as $|\beta_+^j-\beta_-^j|$ grows, and accumulates for
$j=0,\dots,N_\beta-1$. Therefore, it is preferable for $|\beta_-^j-\beta_+^j|$
to be small, so we re-order the recursion parameters such that
$|\beta_+^{j-1}|<|\beta_+^{j}|$ and $|\beta_-^{j-1}|<|\beta_-^{j}|$ for
$j=1,\dots,N_\beta-1$. The error bounds do not depend on the order of the
recursion parameters, but is critical to sort them by magnitude when
implementing OWNS on a computer. To further reduce the impact of rounding
errors, we discard eigenvalues that are large in magnitude (e.g., greater in
magnitude than $\mathcal{O}(100)$).

\subsection{Greedy algorithm and spatial marching}

For systems where $M$ varies slowly in $x$, the recursion parameters computed at
the inlet can be reused at subsequent stations. However, they generally need to
be updated to keep the filter error low, which can be done by \textit{tracking}
them as they evolve downstream (Algorithm~\ref{alg:greedyTracking}). When the
number of left- and right-going characteristics changes, such as in supersonic
boundary-layer flows when the number of subsonic and supersonic grid points
changes, it can be advantageous (although not necessary) to update the
parameters using Algorithm~\ref{alg:greedyRecursion}.

Let $N_+^{(i)}$ and $N_-^{(i)}$ denote the number of left- and right-going
modes, and let $\Xi_{N_\beta}^{+,(i)}$ and $\Xi_{N_\beta}^{-,(i)}$ denote the
greedy recursion parameter sets at station $x^{(i)}$, for $i = 1,\dots,N_x$. The
recursion parameters from station $i-1$ can be used as inputs to a sparse
eigenvalue solver, such as MATLAB's \texttt{eigs} function, to compute the
eigenvalues of $M^{(i)}$ closest to $\Xi_{N_\beta}^{+,(i)}$ and
$\Xi_{N_\beta}^{-,(i)}$. For large, sparse $M$, this approach is far less
expensive than computing the full eigenspectrum.
Algorithm~\ref{alg:greedyTracking} outlines the procedure: if the number of
left- or right-going modes changes, we recompute the recursion parameters from
scratch using Algorithm~\ref{alg:greedyRecursion}; otherwise, we update the
parameters using a sparse eigenvalue solver.

\begin{algorithm}
\caption{Recursion parameter tracking}\label{alg:greedyTracking}
\begin{algorithmic}
\STATE Compute $N_+^{(1)}$ and $N_-^{(1)}$ using Briggs' criterion
(Definition~\ref{def:Briggs})
\STATE Compute $\Xi_{N_\beta}^{+,(1)}$ and $\Xi_{N_\beta}^{-,(1)}$ using
Algorithm~\ref{alg:greedyRecursion}
\FOR{$i = 2,\dots,N_x$}
    \STATE Compute $N_+^{(i)}$ and $N_-^{(i)}$ using Briggs' criterion
    (Definition~\ref{def:Briggs})
    \STATE Initialize $\Xi_0^{+,(i)} \gets \{\},\quad \Xi_0^{-,(i)} \gets \{\}$
    \IF{$N_+^{(i)} = N_+^{(i-1)}$ \textbf{and} $N_-^{(i)} = N_-^{(i-1)}$}
        \FOR{$j = 1,\dots,N_\beta$}
            \STATE Compute $\beta_+^j$ by initializing \texttt{eigs} with
            $(\Xi_{N_\beta}^{+,(i-1)})_j$
            \STATE Compute $\beta_-^j$ by initializing \texttt{eigs} with
            $(\Xi_{N_\beta}^{-,(i-1)})_j$
            \STATE $\Xi_j^{+,(i)} \gets \{\beta_+^j\} \cup \Xi_{j-1}^{+,(i)}$
            \STATE $\Xi_j^{-,(i)} \gets \{\beta_-^j\} \cup \Xi_{j-1}^{-,(i)}$
        \ENDFOR
    \ELSE
        \STATE Compute $\Xi_{N_\beta}^{+,(i)}$ and $\Xi_{N_\beta}^{-,(i)}$ using
        Algorithm~\ref{alg:greedyRecursion}
    \ENDIF
\ENDFOR
\end{algorithmic}
\end{algorithm}

\subsection{Cost trade-off}

Despite the added cost that the greedy algorithm incurs when computing the
eigenvalues of $M$, it can still lead to a net decrease in computational cost,
relative to heuristic selection, since it decreases the number of recursion
parameters required to achieve a desired error tolerance. Moreover, the full
eigenspectrum is required only at a small number of stations, potentially just
the domain inlet, and the selected eigenvalues can be efficiently updated by
\textit{tracking} using a sparse eigenvalues solver. We also note that in
optimization tasks, such as the resolvent analysis by Towne et
al.~\cite{Towne_2022_OWNS-P}, the initial cost of computing recursion parameters
can be amortized over multiple iterations of the optimization routine.

\begin{remark}
    As formulated here, the subset selection problem~\eqref{eq:subsetSelection}
    requires computing the full eigenspectrum of $M$. While feasible for flows
    with a single inhomogeneous direction, this becomes impractical for flows
    with multiple inhomogeneous directions. Future work should therefore focus
    on reformulating~\eqref{eq:subsetSelection} to require only a partial
    eigenspectrum.
\end{remark}

\section{Extension to the Navier-Stokes equations}\label{sec:NSE}

While the Navier-Stokes equations are not hyperbolic, if we neglect the viscous
terms, we obtain the Euler equations, which are hyperbolic. To apply the one-way
marching routine to the Navier-Stokes equations, the streamwise viscous terms
are either neglected or treated as a forcing function~\cite{Towne_2022_OWNS-P,
Kamal_2020_HOWNS,Sleeman_2025_NOWNS_AIAAJ}; we choose to treat them as a forcing
function.

\subsection{Governing equations}

The non-dimensional compressible Navier-Stokes equations for an ideal gas can be
written as
\begin{subequations}
\begin{align}
    \frac{\partial\rho}{\partial t}+\nabla\cdot(\rho\bm{u})
    &=0,\\
    \rho\frac{D\bm{u}}{D t}+\nabla p 
    &= \frac{1}{Re}\nabla\cdot\tau,\\
    \rho c_p\frac{D T}{D t}-\frac{Dp}{Dt}
    &=\frac{1}{Pr Re}\nabla\cdot\big(k(T)\nabla T\big)
    +\frac{1}{Re}\tau :\nabla\bm{u}.
\end{align}
for the stress tensor
\begin{equation}
    \tau = \mu(T) \big(\nabla\bm{u}+(\nabla\bm{u})^T\big)
    -\big(\frac{2}{3}\mu(T)-\kappa(T)\big)(\nabla\cdot\bm{u})I,
\end{equation}
\label{eq:NSE}
\end{subequations}
where $\rho$ is the density, $\bm{u}$ is the velocity, $p$ is the pressure, $T$
is the temperature, $c_p$ is the heat capacity at constant pressure, $\mu$ is
the dynamic viscosity, and $\kappa$ is the bulk viscosity. We take the bulk
viscosity to be zero ($\kappa = 0$) and introduce the Reynolds number,
$Re = \rho_\infty^*a_{\infty}^*\delta_0^* / \mu_\infty^*$, Prandtl number,
$Pr = c_{p,\infty}^*\mu_\infty^* / \kappa_\infty^*$, and Blasius length scale,
$\delta_0^*=\sqrt{\mu_\infty^* x_0^* / \rho_\infty^*u_\infty^*}$. Here,
$(\cdot)_\infty$ and $(\cdot)^*$ denote freestream and dimensional quantities,
respectively, while $\mu$ is the dynamic viscosity, $a$ is the speed of sound,
$u$ is the streamwise velocity, and $x_0$ denotes the inlet boundary. We assume
that the dynamic viscosity is a function of temperature according to
Sutherland's law, while the Prandtl number and specific heat capacity are
constant. We additionally introduce the dimensionless frequency
$F = \omega^*\mu_\infty^* /\rho_\infty^*u_\infty^{*2}$, the dimensionless
wavenumber $b = \beta^*\mu_\infty^* / \rho_\infty^*u_\infty^*$ and the Reynolds
number (based on the streamwise coordinate)
$Re_x = u_\infty^*x^*/\rho_\infty^*\mu_\infty^*$, where $\omega$ is the temporal
frequency and $\beta$ is spanwise wavenumber.

\subsection{Linear OWNS}

Linear stability calculations evolve infinitesimal disturbances, $\bm{q}'$, to a
steady equilibrium solution, $\bar{\bm{q}}$, where $\bm{q}=(\rho,\bm{u},T)$
denotes the vector of primitive variables. Neglecting streamwise viscous terms,
the linear OWNS equations can be written
\begin{equation}
    A(\bar{\bm{q}})\frac{\partial\bm{q}'}{\partial x} = L(\bar{\bm{q}})\bm{q}'
    +\bm{f},
\end{equation}
with $\bm{q}'(x,y,z,t)=\hat{\bm{q}}(x,y)e^{i(\beta z - \omega t)}$ and
$\bm{f}(x,y,z,t)=\hat{\bm{f}}(x,y)e^{i(\beta z - \omega t)}$ where we have
linearized about the steady equilibrium solution and discretized in the
wall-normal direction using a 4th-order central finite difference scheme.

\subsection{Nonlinear OWNS}

If the disturbances have finite amplitude, then disturbances of different
frequency and spanwise wavenumber will interact nonlinearly with each other, so
that we must perform a nonlinear calculation. The nonlinear OWNS equations can
be written
\begin{subequations}
\begin{equation}
    A(\bar{\bm{q}})\frac{\partial\hat{\bm{q}}_{mn}}{\partial x}
    = \hat{L}_{mn}(\bar{\bm{q}})\hat{\bm{q}}_{mn}
    +\hat{\bm{F}}_{mn}(\bm{q}')+\hat{\bm{f}}_{mn},
\end{equation}
for $m=-M,\dots,M$ and $n=-N,\dots,N$, where
\begin{align}
    \bm{q}'(x,y,z,t)
    &=\sum_{m,=-M}^{M}\sum_{m,=-N}^{N}\hat{\bm{q}}_{mn}(x,y)
    e^{i(n\beta z-m\omega t)},\\
    L(\bar{\bm{q}})\bm{q}'
    &=\sum_{m=-M}^{M}\sum_{n=-N}^{N}\hat{L}_{mn}\hat{\bm{q}}_{mn}
    e^{i(n\beta z-m\omega t)},\\
    F(\bm{q}')
    &=\sum_{m=-M}^{M}\sum_{n=-N}^{N}\hat{F}_{mn}(\bm{q}')
    e^{i(n\beta z-m\omega t)}.
\end{align}
\end{subequations}
Note that we truncate the Fourier series at $M$ temporal and $N$ spanwise
Fourier modes, while details on NOWNS are presented by Sleeman et
al.~\cite{Sleeman_2025_NOWNS_AIAAJ}.

\section{Demonstration of the greedy algorithm for a single station}
\label{sec:greedySingle}

We demonstrate the greedy algorithm for the oblique breakdown case studied by
Joslin et al.~\cite{Joslin_1993_DNS} for a low-speed isothermal flat plate
boundary-layer flow with disturbance frequency and wavenumber are
$F = 86 \times 10^{-6}$ and $b = 0.222 \times 10^{-3}$, respectively, for the
wall-normal domain $y\in[0,60]$ with $N_y=100$ at $Re_x=2.74\times10^5$. We find
that greedy parameter selection yields faster error convergence for both OWNS-P
and OWNS-R. We also find that OWNS-R with heuristic parameter selection does not
properly remove all left-going modes, which is related to the rounding errors
due to finite prediction arithmetic (see Remark~\ref{rmk:rounding}).

\subsection{Error convergence}

The projection error scales with the objective function $\mathcal{J}$, while we
can measure it by considering a solution
$\hat{\bm{\phi}}=V_{+}\hat{\bm{\psi}}_++V_{-}\hat{\bm{\psi}}_-$,
where $P\hat{\bm{\phi}}=V_{+}\hat{\bm{\psi}}_+$ by definition. We choose random
coefficients $\hat{\bm{\psi}}$, or coefficients associated with the TS wave, and
compute the relative errors
\begin{subequations}
\begin{align}
    \mathrm{error}(\hat{\bm{\phi}})&
    =\frac{\|\hat{\bm{\phi}}'-P_{N_\beta}\hat{\bm{\phi}}\|}
    {\|\hat{\bm{\phi}}'\|}
    =\frac{\|V_+\bm{\psi}_+-P_{N_\beta}[V_+\bm{\psi}_++V_-\bm{\psi}_-]\|}
    {\|V_+\bm{\psi}_+\|},\label{eq:errPhi}\\
    \mathrm{error}(\hat{\bm{\phi}}_{\mathrm{TS}})&=
    \frac{\|\hat{\bm{\phi}}_{\mathrm{TS}}-P_{N_\beta}
    \hat{\bm{\phi}}_{\mathrm{TS}}\|}{\|\hat{\bm{\phi}}_{\mathrm{TS}}\|}
    =
    \frac{\|V_{\mathrm{TS}}\bm{\psi}_{\mathrm{TS}}-P_{N_\beta}
    [V_{\mathrm{TS}}\bm{\psi}_{\mathrm{TS}}]\|}
    {\|V_{\mathrm{TS}}\bm{\psi}_{\mathrm{TS}}\|}
    \label{eq:errPhiTS}.
\end{align}
\end{subequations}
Figures~\ref{fig:oblique-err-p} and~\ref{fig:oblique-err-r} plot $\mathcal{J}$
against the error~\eqref{eq:errPhi} for greedy and heuristic parameter selection
for OWNS-P and OWNS-R, respectively, showing that the OWNS-P error scales with
$\mathcal{J}$, while the OWNS-R error scales with $\mathcal{J}^{(R)}$ for small
$N_\beta$. Figures~\ref{fig:oblique-err-r} and~\ref{fig:oblique-err-TS-r} show
that the OWNS-R error, when using greedy selection increases when $N_\beta$
grows large, as discussed in Remark~\ref{rmk:rounding}. In contrast, the error
when using heuristic selection for OWNS-R is non-increasing with increasing
$N_\beta$, suggesting that heuristic selection is less prone to rounding errors
than greedy selection.

Figure~\ref{fig:oblique-err-p} shows that greedy OWNS-P achieves machine zero
error at $N_\beta=24$, while heuristic OWNS-P converges far more slowly and does
not achieve machine zero error, even for much larger $N_\beta$. Greedy OWNS-R
approaches machine zero error for $N_\beta\approx 40$, after which point the
error increases due to rounding errors. We further see in
Figure~\ref{fig:oblique-err-r} that OWNS-R with heuristic selection leads to
large errors, even with large $N_\beta$.

Figure~\ref{fig:oblique-err-TS-p} shows that greedy selection yields
significantly less error in the TS wave than heuristic selection when using
OWNS-P. In contrast, for OWNS-R, Figure~\ref{fig:oblique-err-TS-r} shows that
greedy selection yields less error for small $N_\beta$, while heuristic
selection yields less error for large $N_\beta$, consistent with our
observations in Figure~\ref{fig:oblique-err-r}. We further note that by
Propositions~\ref{prop:greedy-p}, the error in the TS waves goes to zero if it
is used as a recursion parameter, and the sudden decrease in error occurs when
the greedy algorithm selects the TS wave as a recursion parameter.

Figure~\ref{fig:greedy-err-converge-quad} compares the error obtained using quad
and double precision to compute $\{\beta_*^j\}_{j=1}^{N_\beta}$ for OWNS-R with
heuristic parameter selection. Double precision yields rapidly increasing error
as $N_\beta$ grows. Results for greedy selection are omitted, as quad precision
provides negligible improvement in that case, likely due to the larger magnitude
of the greedy parameters. For $N_\beta < 45$, double precision does not
significantly increase the OWNS-R error, indicating that it is adequate in this
regime. Therefore, we use double precision in the calculations presented in
Section~\ref{sec:greedyMarching}.

Figure~\ref{fig:param-spectrum-heuristic} shows the heuristic parameters for
$N_\beta=20$. We see that although heuristic selection attempts to place its
parameters near the TS wave and along the continuous branches, it does so less
efficiently than the greedy algorithm. Although it accurately targets the
left-going acoustic waves along the imaginary axis for $\mathcal{I}(\alpha)<0$,
it distributes the parameters for the right-going acoustic waves (near the
imaginary axis for $\mathcal{I}(\alpha)>0$) and the vortical waves (along the
axis at $\pi/4$ radians for $\mathcal{I}(\alpha)>0$) less accurately. Moreover,
it places parameters along the imaginary axis in the left half-plane, where
there are no eigenvalues. Thus, the greedy parameters suggest avenues for
improving the heuristic selection procedure.

\begin{remark}
    In Section~\ref{subsec:OWNSR}, our numerical analysis showed that OWNS-P
    will converge faster than OWNS-R, which is what we observe empirically in
    Figure~\ref{fig:greedy-err-converge}.
\end{remark}

\begin{figure}
    \centering
    \begin{subfigure}[b]{0.48\textwidth}
        \centering
        \includegraphics[width=1\linewidth]{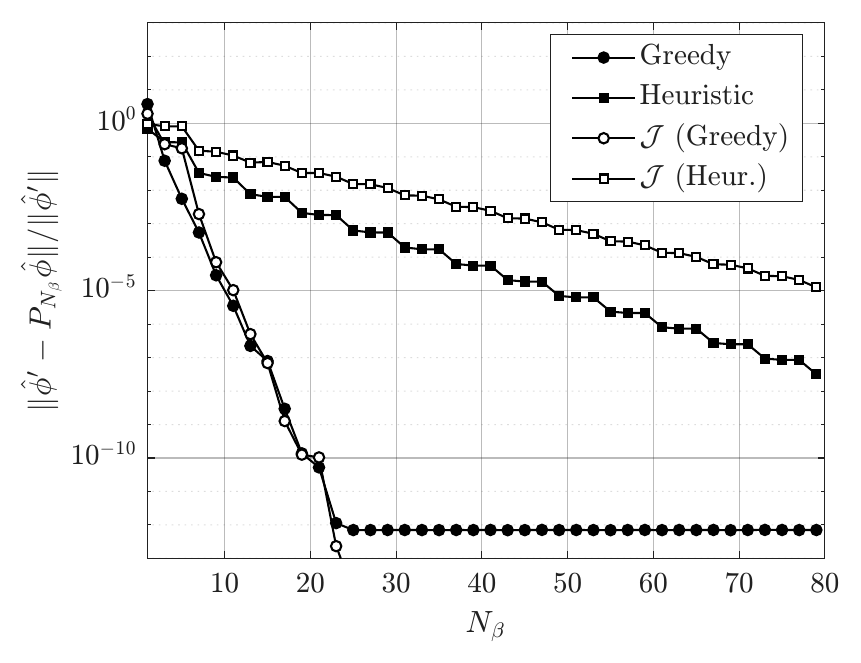}
        \caption{OWNS-P, random modes}\label{fig:oblique-err-p}
    \end{subfigure}
    \begin{subfigure}[b]{0.48\textwidth}
        \centering
        \includegraphics[width=1\linewidth]{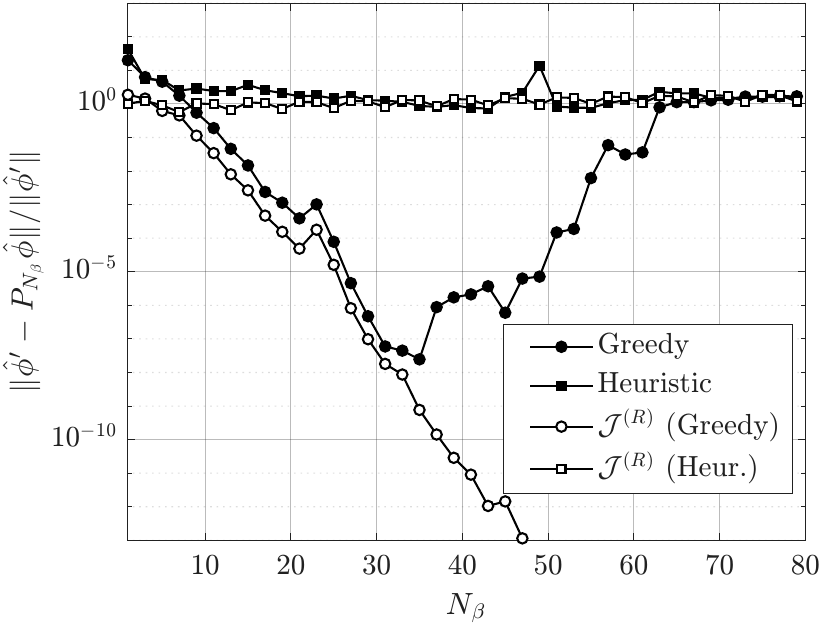}
        \caption{OWNS-R, random modes}\label{fig:oblique-err-r}
    \end{subfigure}\\
    \begin{subfigure}[b]{0.48\textwidth}
        \centering
        \includegraphics[width=1\linewidth]{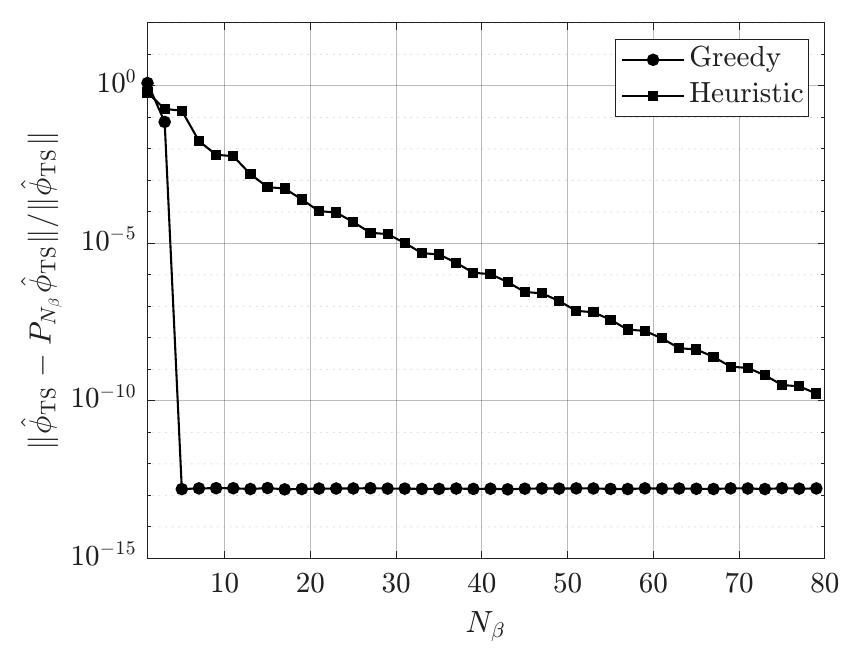}
        \caption{OWNS-P, TS wave only}\label{fig:oblique-err-TS-p}
    \end{subfigure}
    \begin{subfigure}[b]{0.48\textwidth}
        \centering
        \includegraphics[width=1\linewidth]{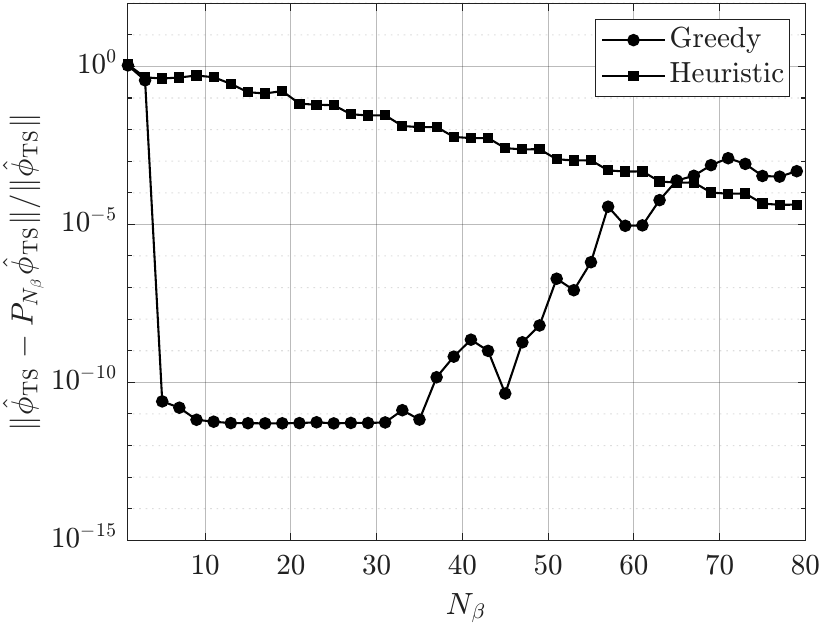}
        \caption{OWNS-R, TS wave only}\label{fig:oblique-err-TS-r}
    \end{subfigure}\\
    \caption{Error convergence for low-speed oblique breakdown, comparing OWNS-P
    and OWNS-R, using quad precision to compute $\{\beta_*^j\}_{j=1}^{N_\beta}*$
    for OWNS-R.}
    \label{fig:greedy-err-converge}
\end{figure}

\begin{figure}
    \centering
    \begin{subfigure}[b]{0.48\textwidth}
        \centering
        \includegraphics[width=1\linewidth]{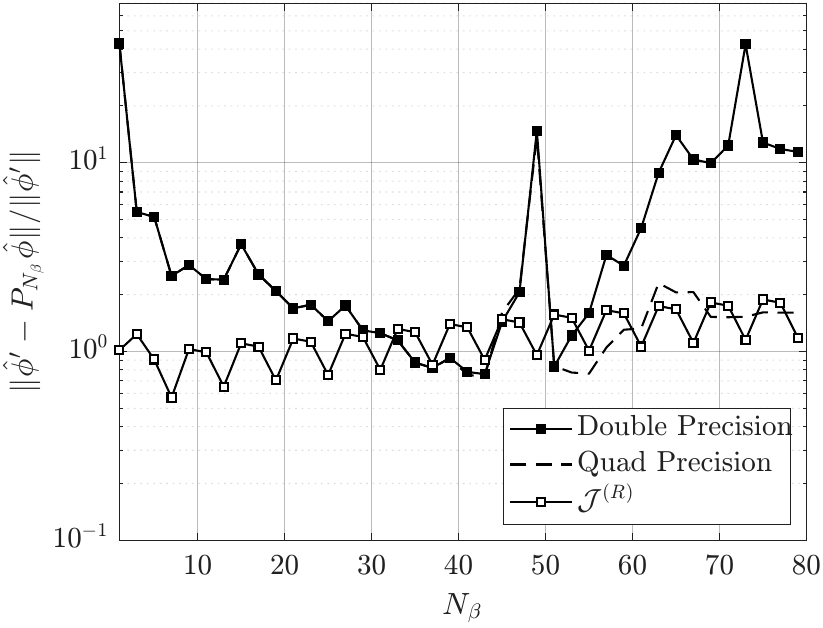}
        \caption{Random modes}\label{fig:oblique-err-r-quad}
    \end{subfigure}\hfill
    \begin{subfigure}[b]{0.48\textwidth}
        \centering
        \includegraphics[width=1\linewidth]
        {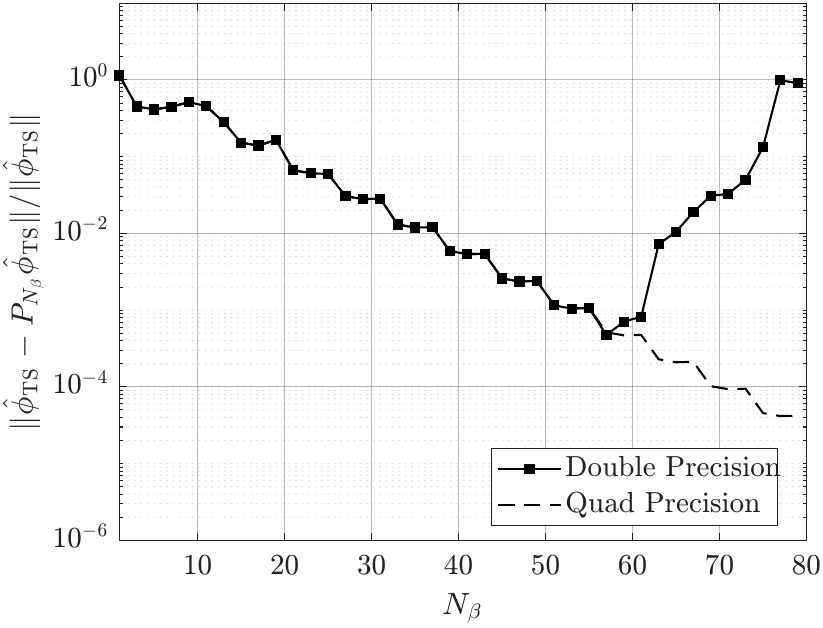}
        \caption{TS wave only}\label{fig:oblique-err-r-TS-quad}
    \end{subfigure}
    \caption{Error convergence for low-speed oblique breakdown for OWNS-R,
    comparing using double and quad precision to compute
    $\{\beta_*^j\}_{j=1}^{N_\beta}*$ for heuristic recursion parameter selection
    only.}
    \label{fig:greedy-err-converge-quad}
\end{figure}

\begin{figure}
    \centering
    \begin{subfigure}[b]{0.48\textwidth}
        \centering
        \includegraphics[width=1\linewidth]
        {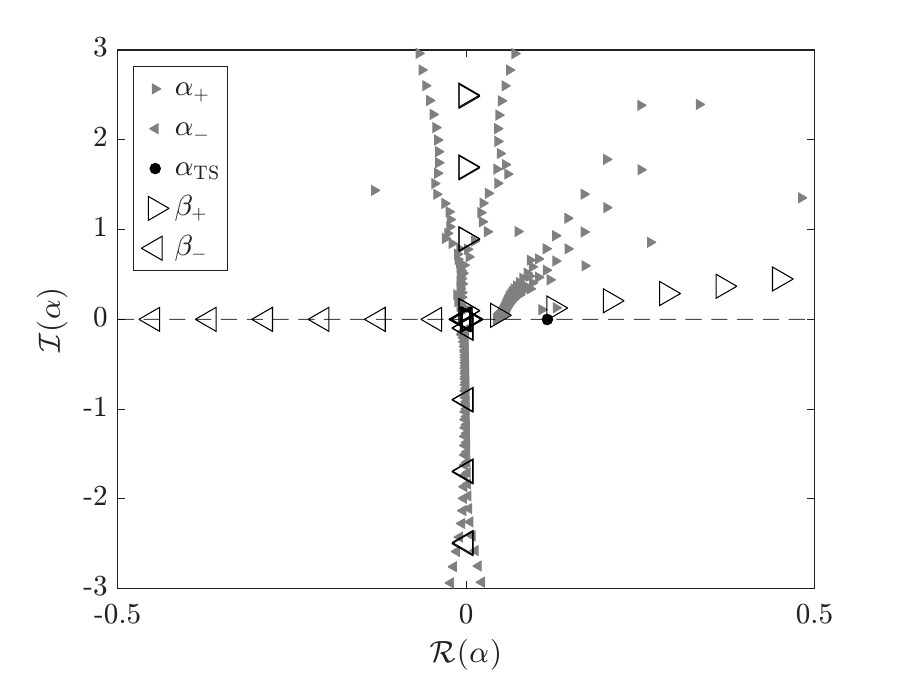}
        \caption{Heuristic}
        \label{fig:param-spectrum-heuristic}
    \end{subfigure}
    \begin{subfigure}[b]{0.48\textwidth}
        \centering
        \includegraphics[width=1\linewidth]
        {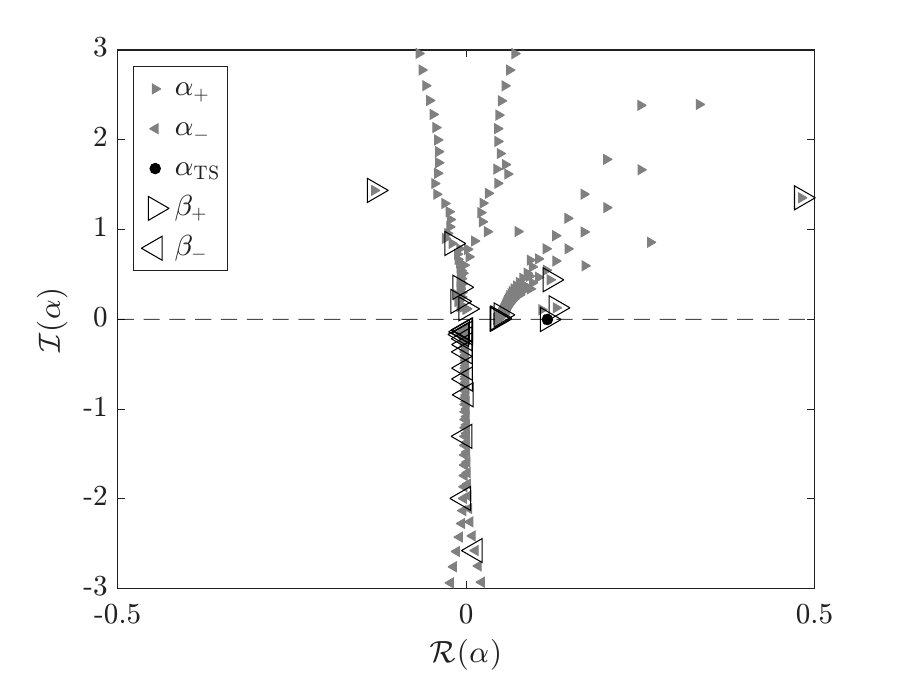}
        \caption{Greedy}
        \label{fig:param-spectrum-greedy}
    \end{subfigure}
    \caption{Recursion parameters plotted against spectrum for heuristic and
    greedy recursion parameter selection with $N_\beta = 20$.}
    \label{fig:param-spectrum}
\end{figure}

\subsection{Stability of the march}\label{subsec:stability}

All left-going modes must be removed for the march to be stable.
Figure~\ref{fig:Oblique_OWNS_P_Briggs} shows that heuristic OWNS-P removes all
left-going waves from $M$, while Figure~\ref{fig:Oblique_OWNS_R_Briggs}
demonstrates that heuristic OWNS-R does not. For $\eta=0$, we observe that the
OWNS-R spectrum, $P_{N_\beta}^{(R)}M$, has many eigenvalues with
$\mathcal{I}(\alpha)<0$, and taking $\eta=1000$ allows us to verify that these
are left-going modes according to Briggs' criterion. This demonstrates that
there exists parameter sets for which OWNS-P is stable but OWNS-R is unstable.
Repeating this analysis with greedy selection yields stable marches for both
OWNS-P and OWNS-R.

\begin{remark}
Figures 3 and 6 from Zhu and Towne~\cite{Zhu_2021_OWNS-R} show that the OWNS-R
operator removes all left-going modes for the dipole and jet test cases, while
Figure 10 suggests it still supports left-going modes for their supersonic
boundary-layer case. Thus, our observations in
Figure~\ref{fig:Oblique_OWNS_R_Briggs} are consistent with those of Zhu and
Towne~\cite{Zhu_2021_OWNS-R} for their supersonic boundary-layer flow case.
\end{remark}

\begin{figure}
    \centering
    \begin{subfigure}[b]{0.48\textwidth}
        \centering
        \includegraphics[width=1\linewidth]
        {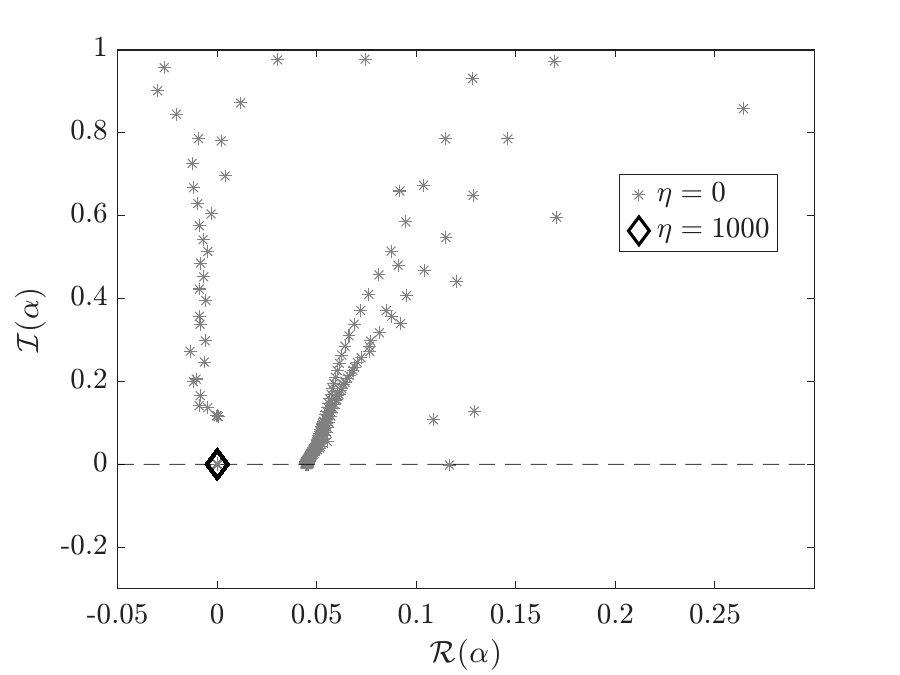}
        \caption{Heuristic OWNS-P, $N_\beta=15$}
        \label{fig:Oblique_OWNS_P_Briggs}
    \end{subfigure}
    \begin{subfigure}[b]{0.48\textwidth}
        \centering
        \includegraphics[width=1\linewidth]
        {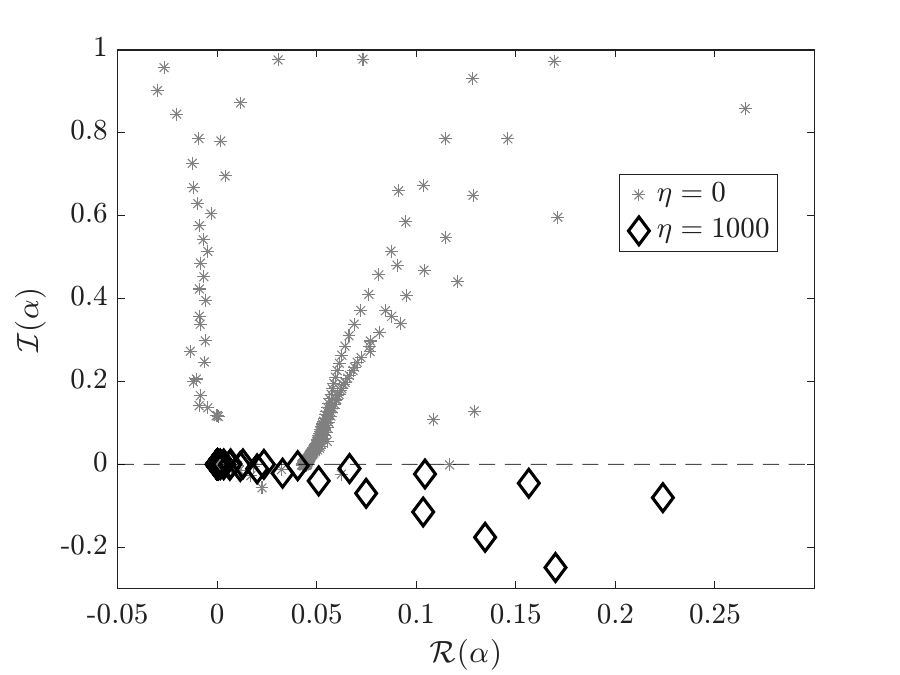}
        \caption{Heuristic OWNS-R, $N_\beta=52$}
        \label{fig:Oblique_OWNS_R_Briggs}
    \end{subfigure}\\
    \caption{For the march to be stable, the OWNS spectrum ($P_{N_\beta}M$) must
    not support any left-going modes. According to Briggs' criterion, a mode is
    left-going if $\mathcal{I}(\alpha)<0$ as $\eta\to\infty$.
    Figure~\ref{fig:Oblique_OWNS_P_Briggs} shows that heuristic OWNS-P does not
    have $\mathcal{I}(\alpha)<0$ for $\eta=1000$,
    while~\ref{fig:Oblique_OWNS_R_Briggs} shows that heuristic OWNS-R does, so
    that the heuristic OWNS-R march is unstable.}
    \label{fig:greedy-stability}
\end{figure}

\section{Demonstration of the greedy algorithm for spatial marching}
\label{sec:greedyMarching}

Here we demonstrate that although the greedy algorithm must compute the
eigenvalues of $M$, it leads to a net decrease in computational cost for
linear and nonlinear disturbance evolution in low- and high-speed boundary-layer
flows. In addition, we show that greedy selection alleviates numerical errors
introduced by heuristic parameter selection for nonlinear disturbance evolution
in a high-speed boundary-layer flow.

\subsection{3D low-speed oblique breakdown}

We first demonstrate for the  oblique breakdown case discussed in
Section~\ref{sec:greedySingle} that greedy parameter selection reduces
computational cost in both linear and nonlinear calculations. The problem
parameters remain the same as before, and we consider the streamwise domain
$Re_x \in [2.74, 6.08] \times 10^5$ with $N_x=2000$ stations. We compare the
number of recursion parameters required for convergence of the relative error in
the $N$-factor, which is a measure of the disturbance growth used in the $e^N$
method to empirically predict laminar-turbulent transition in boundary-layer
flows~\cite{VanIngen_1956_eN,Smith_1956_eN}, and measure speed-up based on total
wall-clock time. Note that the $N$-factor is defined as
$N = \max_{x} \ln [ A(x) / A(x_0) ]$, where $A(x) = \max_{y} u'(x,y)$ is the
disturbance amplitude and $x_0$ is domain inlet.

\subsubsection{Linear calculation}

OWNS-R with heuristic parameter selection does not fully remove left-going
waves, so the march is unstable (see Figure~\ref{fig:greedy-stability}).
Therefore, we compare OWNS-P and OWNS-R using greedy selection against OWNS-P
using heuristic selection with and without mode-tracking. To make meaningful
comparisons, we take $N_\beta$ to be the smallest value that yields 10\%, 1\%,
and 0.1\% relative error in the $N$-factor, measured against the greedy OWNS-P
solution obtained using $N_\beta = 30$. We further compare the performance of
the greedy algorithm when it is performed at every step of the march against its
performance when it is applied only at the domain inlet, with the recursion
parameters updated for subsequent stations according to
Algorithm~\ref{alg:greedyTracking}. We summarize these results in
Table~\ref{tab:oblique-lin-err}, where speed-up is measured relative to the
naive heuristic OWNS-P calculation (P-H) at each error level. We see that OWNS-R
with greedy selection applied only at the inlet (R-GI) leads to the largest
speed-up relative to P-H. Heuristic mode-tracking OWNS-P (P-T) and OWNS-R with
greedy selection at all stations (R-GA) lead to similar speed-ups compared to
each other, while OWNS-P with greedy selection only at the inlet (P-GI) yields
slightly lower speeds-up than P-T and R-GA. Finally, OWNS-P with greedy
selection at all stations (P-GA) yields the least speed-up.
Table~\ref{tab:oblique-lin-err} clearly demonstrates both the higher speeds of
OWNS-R and the better convergence of OWNS-P, which requires fewer recursion
parameters when using greedy selection.

\begin{table}[h]
\centering
\begin{subtable}{\linewidth}
\centering
\caption{$N_\beta$ required to reach target error}
\begin{tabular}{c | cccccc}
\toprule
Error & P-H & P-T & P-GA & P-GI & R-GA & R-GI \\
\midrule
10\%  & 11 & 5 & 5 & 6 & 8  & 8 \\
1\%   & 17 & 6 & 6 & 6 & 11 & 11 \\
0.1\% & 18 & 6 & 6 & 6 & 16 & 17 \\
\bottomrule
\end{tabular}
\end{subtable}
\\
\begin{subtable}{\linewidth}
\centering
\caption{Speed-up relative to heuristic OWNS-P (P-H)}
\begin{tabular}{c | cccccc}
\toprule
Error & P-H & P-T & P-GA & P-GI & R-GA & R-GI \\
\midrule
10\%  & 1.0 & 2.6 & 1.4 & 2.0 & 2.2 & 8.0 \\
1\%   & 1.0 & 3.6 & 2.0 & 3.4 & 3.8 & 12.5 \\
0.1\% & 1.0 & 3.8 & 2.1 & 3.6 & 3.8 & 10.6 \\
\bottomrule
\end{tabular}
\end{subtable}
\caption{Speed-up for 3D low-speed oblique-wave breakdown using different
recursion-parameter selection strategies, relative to heuristic OWNS-P (P-H):
mode-tracking (with heuristic selection) for the TS wave with OWNS-P (P-T);
greedy selection at all streamwise stations with OWNS-P (P-GA) and OWNS-R
(R-GA); and greedy selection at the inlet with recursion-parameter tracking
using OWNS-P (P-GI) and OWNS-R (R-GI).}
\label{tab:oblique-lin-err}
\end{table}

\subsubsection{Nonlinear calculation}

For the nonlinear calculations, we truncate the Fourier series at $M=3$ temporal
and $N=4$ spanwise modes, while we specify at the inlet that the oblique wave
has amplitude 0.141\% based on the free-stream $u$-velocity. We select $N_\beta$
to be the smallest value such that the relative error in the $N$-factor is
0.1\% with respect to the reference solution (greedy selection with
$N_\beta=30$). For NOWNS, heuristic and greedy parameter selection require
$N_\beta = 18$ and $N_\beta = 6$, respectively, resulting in a speed-up of 3.83
for the greedy approach.

\subsection{2D high-speed Mack's second mode}

We apply NOWNS to the Mach 4.5 boundary-layer flow over an adiabatic flat plate
studied by Ma and Zhong~\cite{Ma_2003_DNS_1}, where $M_\infty = 4.5$,
$T_\infty^*=65.15$ K, $p_\infty^*=728.44$ Pa, $Pr=0.72$, and
$Re_\infty^* = \rho_\infty^* U_\infty^* / \mu_\infty^*=7.2\times10^6 / $ m.
For the disturbance frequency $F=2.2\times10^{-4}$, Mack's second mode (MM) is
the dominant instability. We truncate the Fourier series at $M=10$ modes and
specify that MM have an amplitude of 10\% based on the free-stream temperature.
We compare the temperature profiles obtained using greedy and heuristic
parameter selection in Figure~\ref{fig:NOWNS-Zhu-01}, and discuss how heuristic
NOWNS introduces spurious numerical fluctuations along the sonic line.

Figure~\ref{fig:NOWNS-Zhu-T1-0500} shows that greedy and heuristic parameter
selection predict similar temperature profiles for $\hat{T}_1'$ near the inlet
($\hat{T}_m'$ refers to the temperature profile for mode $m$), while
Figure~\ref{fig:NOWNS-Zhu-T1-2501} shows that at the domain outlet, heuristic
NOWNS predicts a feature along the sonic line that is not predicted by greedy
NOWNS. Figure~\ref{fig:NOWNS-Zhu-T3-0500} shows that near the domain inlet,
heuristic NOWNS predicts a large fluctuation in $\hat{T}_3'$ that is not
predicted by greedy NOWNS. As the solution is marched downstream, this feature
corrupts the solution, eventually leading to the fluctuation in $\hat{T}_1'$
observed for heuristic NOWNS in Figure~\ref{fig:NOWNS-Zhu-T1-2501}. Although
heuristic selection leads to accurate results for linear calculations for
high-speed boundary-layer flows~\cite{Kamal_2020_HOWNS},
Figure~\ref{fig:NOWNS-Zhu-01} shows that the results become inaccurate for
nonlinear calculations, and that greedy selection overcomes this challenge.

\begin{figure}
    \centering
    \begin{subfigure}[b]{0.48\textwidth}
        \centering
        \includegraphics[width=1\linewidth]{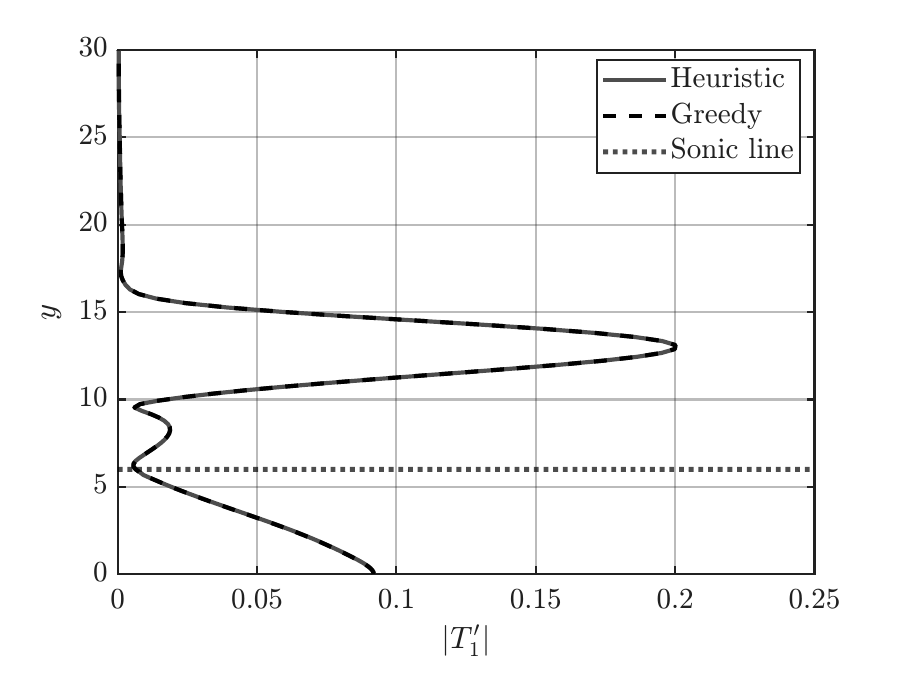}
        \caption{$\hat{T}_1'$ at $Re_x=6.92\times10^5$}
        \label{fig:NOWNS-Zhu-T1-0500}
    \end{subfigure}
    \begin{subfigure}[b]{0.48\textwidth}
        \centering
        \includegraphics[width=1\linewidth]{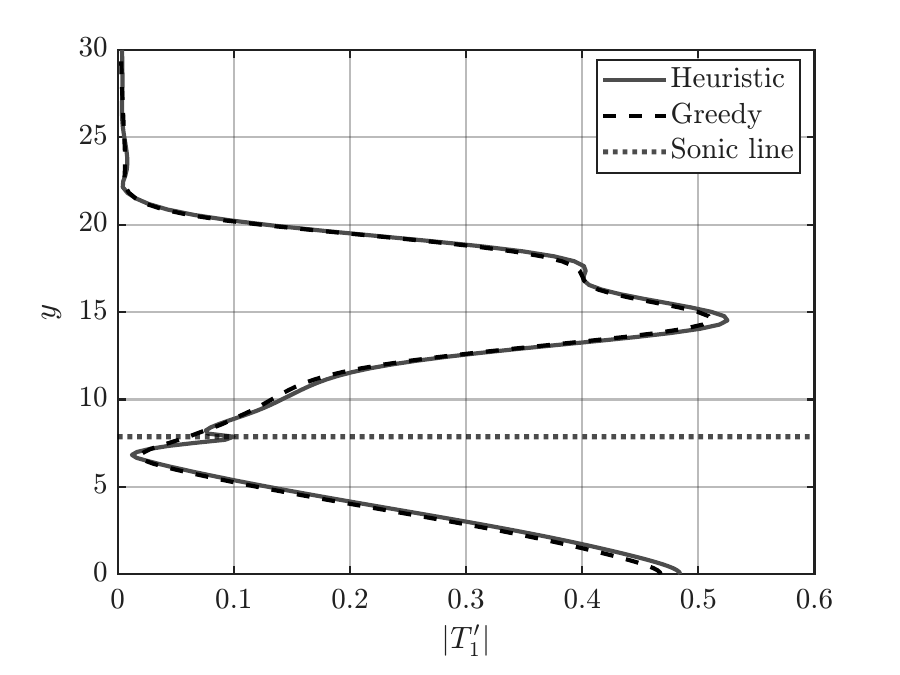}
        \caption{$\hat{T}_1'$ at $Re_x=1.21\times10^6$}
        \label{fig:NOWNS-Zhu-T1-2501}
    \end{subfigure}\\
    \begin{subfigure}[b]{0.48\textwidth}
        \centering
        \includegraphics[width=1\linewidth]{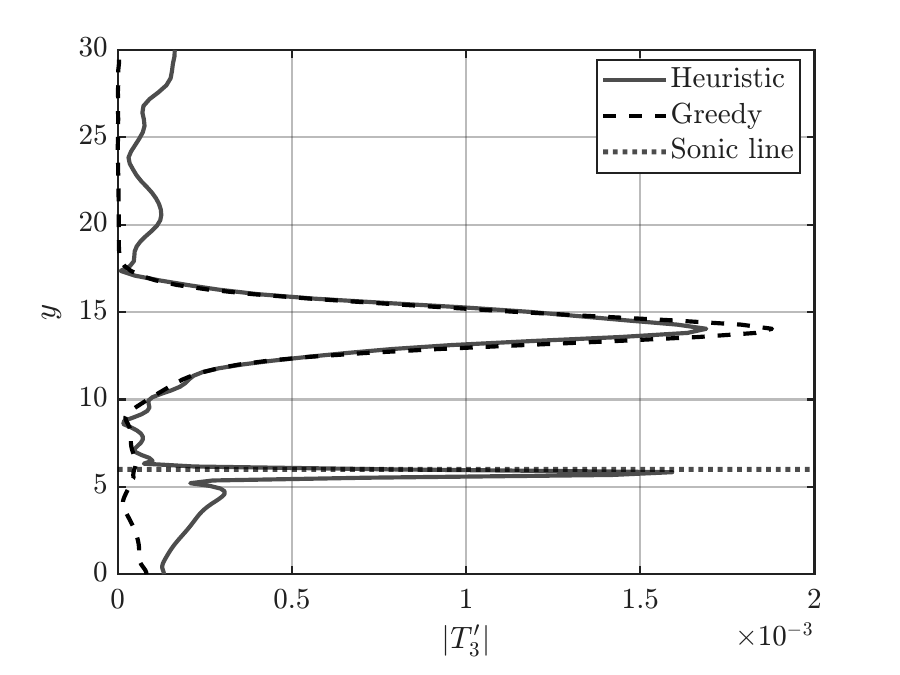}
        \caption{$\hat{T}_3'$ at $Re_x=6.92\times10^5$}
        \label{fig:NOWNS-Zhu-T3-0500}
    \end{subfigure}
    \begin{subfigure}[b]{0.48\textwidth}
        \centering
        \includegraphics[width=1\linewidth]{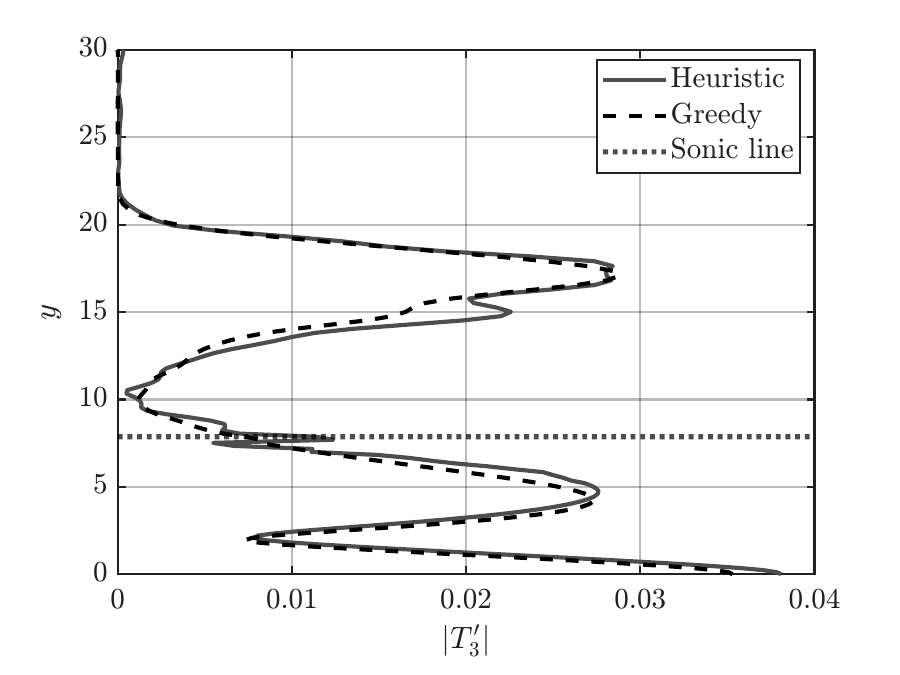}
        \caption{$\hat{T}_3'$ at $Re_x=1.21\times10^6$}
        \label{fig:NOWNS-Zhu-T3-2501}
    \end{subfigure}    
    \caption{Temperature profiles for modes $\hat{T}_1'$ and $\hat{T}_3'$
    computed for nonlinear evolution of Mack's second mode.}
    \label{fig:NOWNS-Zhu-01}
\end{figure}

\section{Conclusion}

Through numerical analyses and experiments, we have compared the convergence
properties of the recursive filters used by OWNS-P and OWNS-R, and introduced a
\textit{subset selection} problem, which we solve using a greedy algorithm, to
automatically choose recursion parameters that ensure rapid convergence of the
filter error. Although we have used the OWNS label, which refers to the method
when generalized to the linearized Navier-Stokes equations, we would like to
emphasize that all of the analyses also apply to systems of linear first-order
hyperbolic equation, and that the greedy algorithm could also be applied to such
systems. In summary, we showed analytically that OWNS-P has better stability and
accuracy properties and that there exist recursion parameter sets for which
OWNS-P is fully converged, while OWNS-R is not. Conversely, if OWNS-R is fully
converged, then OWNS-P must also be fully converged. In particular, we
demonstrated numerically for boundary-layer flows that there exists heuristic
parameter sets that yield a stable and accurate OWNS-P march, while the OWNS-R
march is unstable. We observe that when using heuristic parameter selection,
increasing $N_\beta$ does not yield stable and accurate OWNS-R marches due to
rounding errors associated with finite precision arithmetic introduced when
applying the OWNS-R filter~\eqref{eq:filter-r} as $N_\beta$ grows large, while
greedy selection overcomes this challenge by keeping $N_\beta$ small.
Table~\ref{tab:oblique-lin-err} demonstrates that OWNS-R leads to substantial
speed-up relative to OWNS-P when convergent recursion parameters are chosen,
while Zhu and Towne~\cite{Zhu_2021_OWNS-R} demonstrated that OWNS-R scales
better for larger systems of equations. Future work should investigate how to
choose better recursion parameters for OWNS-R so that $N_\beta$ remains small,
either through the use of the greedy algorithm presented here or through
heuristic recursion parameter selection.

We demonstrated for linear and nonlinear disturbance evolution in boundary-layer
flows that our proposed greedy algorithm yields faster error convergence and a
reduced computational cost relative to naive heuristic parameter selection. In
the linear case, greedy OWNS-P entails a higher cost than heuristic OWNS-P with
mode-tracking, while greedy OWNS-R entails a lower cost than any of the OWNS-P
calculations. We further demonstrated the heuristic NOWNS for high-speed
boundary-layer flows introduces spurious numerical oscillations, which are
eliminated using the greedy algorithm, suggesting that the existing heuristic
routine is not sufficiently accurate for nonlinear calculations. We compared the
greedy and heuristic recursion parameter sets, and highlighted that greedy OWNS
converges more quickly since it places its eigenvalues more efficiently, and we
note that future work could use greedy parameter selection to help tune the
heuristic routines. We further emphasize that the heuristic parameter selection
routine requires an analytical approximation to the eigenvalues of the governing
equations (e.g., the eigenvalues of the Euler equations linearized about a
uniform flow can be computed analytically) and a method to use this
approximation to construct sets of convergence recursion parameters. For
equations of interest, an analytical approximation to the eigenvalues may not
exist, and even if it does, it may not be obvious how to use the expression to
construct convergent recursion parameters sets. The greedy algorithm allows for
automatic recursion parameter selection without the need for analytical
approximations, which will allow the OWNS approach to be applied more easily to
new systems of linear first-order hyperbolic equations. Finally, we note that
this work focused only on flows with a single inhomogeneous direction (the
wall-normal direction), and we note that the approach may not extend readily to
flows with two or more inhomogeneous directions due to the increased grid size,
and so future work should explore how to adapt the greedy algorithm for such
cases so that the full eigenspectrum of $M$ need not be computed.

\section*{Acknowledgments}

We acknowledge support from The Boeing Company
\linebreak
through the Strategic Research
and Development Relationship Agreement CT-BA-GTA-1, and by the Office of Naval
Research through N00014-25-1-2072 and N00014-21-1-2158.

\appendix
\section{Proofs}~\label{app:Proofs}

Proposition~\ref{prop:minimalOWNS-P} was proved by Towne et
al.~\cite{Towne_2022_OWNS-P}, restated here in our notation.
\begin{proof}[Proof of Proposition~\ref{prop:minimalOWNS-P}]
    If there exists $(\hat{m},\hat{n})\in i^{(+)}\times i^{(-)}$ such that
    $\alpha_{\hat{m}}=\alpha_{\hat{n}}$, then
    \[
    R_{\hat{n}\hat{m}}
    =
    \prod_{j=0}^{N_\beta-1}\frac{\alpha_{\hat{m}}-\beta_{+}^{j}}
    {\alpha_{\hat{m}}-\beta_-^j}\frac{\alpha_{\hat{n}}-\beta_{-}^{j}}
    {\alpha_{\hat{n}}-\beta_+^j}
    (V_{--}^{-1}V_{-+})_{\hat{n}\hat{m}}
    =
    (V_{--}^{-1}V_{-+})_{\hat{n}\hat{m}}
    \neq 0
    \]
    so that $R_{N_\beta}^{-1}\neq I$ for any choice of recursion parameters.

    Next assume that $\alpha_m\neq\alpha_n$ for all $(m,n)\in i^{(+)}\times
    i^{(-)}$, and note that
    \[
    \max_{(m,n)\in i^{(+)}\times i^{(-)}}\prod_{j=0}^{N_\beta-1}
    \frac{|\alpha_m-\beta_{+}^{j}|}
    {|\alpha_m-\beta_-^j|}
    \frac{|\alpha_n-\beta_{-}^{j}|}
    {|\alpha_n-\beta_+^j|}=\|F_{++}\|\|F_{--}^{-1}\|,
    \]
    so that by Proposition~\ref{prop:approxProjectConvergence},
    $P_{N_\beta}\to P$ when $\|F_{++}\|\|F_{--}^{-1}\|=0$. If
    $\tilde{N}_+\leq \tilde{N}_-$, we use~\eqref{eq:owns-p-params_plus} to
    obtain $\|F_{++}\|=0$ since $\alpha_m\in\{\beta_j\}_{j=0}^{N_\beta-1}$ and
    $\alpha_m\notin\{\beta_j\}_{j=0}^{N_\beta-1}$ for all $m \in i^{(+)}$. If
    $\tilde{N}_- < \tilde{N}_+$, we use~\eqref{eq:owns-p-params_minus} to obtain
    $\|F_{--}^{-1}\|=0$ since $\alpha_n\in\{\beta_-^j\}_{j=0}^{N_\beta-1}$ and
    $\alpha_n\notin\{\beta_+^j\}_{j=0}^{N_\beta-1}$ for all $n\in i^{(-)}$.
\end{proof}

The proof of Proposition~\ref{prop:errOWNSP}, which bounds the error
$\|P-P_{N_\beta}\|$ for small $\|F_{++}\|\|F_{--}^{-1}\|$, uses
Proposition~\ref{prop:errBoundTriangle} to bound
$\|E-R_{N_\beta}^{-1}ER_{N_\beta}\|$.

\begin{proposition}\label{prop:errBoundTriangle}
    The norm of $E-R_{N_\beta}ER_{N_\beta}^{-1}$ can be bounded
    \begin{align}
    \begin{split}
        \|E-R_{N_\beta}ER_{N_\beta}^{-1}\|&\leq
        \|(I_{++}-F_{++}V_{++}^{-1}V_{+-}F_{--}^{-2}
        V_{--}^{-1}V_{-+}F_{++})^{-1}\|\\
        &\big(
        2\|F_{++}\|^2\|F_{--}^{-1}\|^2\|V_{++}^{-1}V_{+-}\|
        \|V_{--}^{-1}V_{-+}\|\\
        &+\|F_{++}\|\|F_{--}^{-1}\|\|V_{++}^{-1}V_{+-}\|+\|F_{++}\|
        \|F_{--}^{-1}\|\|V_{--}^{-1}V_{-+}\|\|
        \big).
    \end{split}
    \label{eq:errBoundTriangleFinal}
    \end{align}
\end{proposition}
\begin{proof}
    $R_{N_\beta}$ comprises four blocks, where the diagonal blocks are square
    and the off-diagonal block are conformal with them for partitioning. Recall
    from~\eqref{eq:approxProjectRMat} that $R_{N_\beta}$ is defined as the
    inverse of another matrix, so that obtaining a simplified expression for
    $R_{N_\beta}^{-1}$ is straightforward, while it is more challenging for
    $R_{N_\beta}$. Define $S=R_{N_\beta}^{-1}$, so that it can be inverted
    blockwise as
    \[
        R_{N_\beta}=S^{-1}=
        \begin{bmatrix}
            S_{++} & S_{+-}\\
            S_{-+} & S_{--}
        \end{bmatrix}^{-1}
        =
        \begin{bmatrix}
            [S^{-1}]_{++} & [S^{-1}]_{+-}\\
            [S^{-1}]_{-+} & [S^{-1}]_{--}
        \end{bmatrix}.
    \]
    Then
    \[
    E-R_{N_\beta}E R_{N_\beta}^{-1}
    =
    E-S^{-1} E S
    =
    \begin{bmatrix}
            ([S^{-1}]_{++})(S_{++})- I_{++} &
            ([S^{-1}]_{++})(S_{+-})\\
            ([S^{-1}]_{-+})(S_{++}) &
            ([S^{-1}]_{-+})(S_{+-})
        \end{bmatrix}.
    \]
    and by the triangle inequality
    \begin{align}
    \begin{split}
    \|E-R_{N_\beta}^{-1}E R_{N_\beta}\|
    &\leq
    \|([S^{-1}]_{++})(S_{++}) - I_{++}\|
    +\|([S^{-1}]_{++})(S_{+-})\|\\
    &+\|([S^{-1}]_{-+})(S_{++})\|
    +\|([S^{-1}]_{-+})(S_{+-})\|.
    \end{split}
    \label{eq:errBoundTriangle}
    \end{align} 
    Next we have
    \begin{align*}
        [S^{-1}]_{++}
        &=(S_{++}-S_{+-}S_{--}^{-1}S_{-+})^{-1},\\
        &=(I_{++}-F_{++}V_{++}^{-1}V_{+-}F_{--}^{-2}
        V_{--}^{-1}V_{-+}F_{++})^{-1}\\
        [S^{-1}]_{-+}
        &=-S_{--}^{-1}S_{-+}[S^{-1}]_{++}\\
        &=-F_{--}^{-1}V_{--}^{-1}V_{-+}F_{++}(I_{++}-F_{++}V_{++}^{-1}
        V_{+-}F_{--}^{-2}V_{--}^{-1}V_{-+}F_{++})^{-1},
    \end{align*}
    while $S_{++}=I_{++}$ and $S_{-+}=F_{++}V_{++}^{-1}V_{+-}F_{--}^{-1}$ so
    that
    \begin{subequations}
    \begin{align}
    \begin{split}
        \|([S^{-1}]_{++})(S_{++})- I_{++}\|
        &\leq
        \|F_{++}V_{++}^{-1}V_{+-}F_{--}^{-1}\|\times
        \|F_{--}^{-1}V_{--}^{-1}V_{-+}F_{++}\|\\
        &\times
        \|(I_{++}-F_{++}V_{++}^{-1}V_{+-}F_{--}^{-2}V_{--}^{-1}
        V_{-+}F_{++})^{-1}\|
    \end{split}\\
    \begin{split}
        \|([S^{-1}]_{++})(S_{+-})\|
        &\leq
        \|F_{++}V_{++}^{-1}V_{+-}F_{--}^{-1}\|\\
        &\times
        \|(I_{++}-F_{++}V_{++}^{-1}V_{+-}F_{--}^{-2}V_{--}^{-1}V_{-+}
        F_{++})^{-1}\|,
    \end{split}\\
    \begin{split}
        \|([S^{-1}]_{-+})(S_{++})\|
        &\leq\|F_{--}^{-1}V_{--}^{-1}V_{-+}F_{++}\|\\
        &\times\|(I_{++}-F_{++}V_{++}^{-1}V_{+-}F_{--}^{-2}V_{--}^{-1}V_{-+}
        F_{++})^{-1}\|,
    \end{split}\\
    \begin{split}
        \|([S^{-1}]_{-+})(S_{+-})\|
        &\leq\|F_{--}^{-1}V_{--}^{-1}V_{-+}F_{++}\|
        \times\|F_{++}V_{++}^{-1}V_{+-}F_{--}^{-1}\|\\
        &\times\|(I_{++}-F_{++}V_{++}^{-1}V_{+-}F_{--}^{-2}V_{--}^{-1}V_{-+}
        F_{++})^{-1}\|.
    \end{split}
    \end{align}
    \label{eq:errBoundBlockInverse}
    \end{subequations}
    Substituting~\eqref{eq:errBoundBlockInverse}
    into~\eqref{eq:errBoundTriangle} yields~\eqref{eq:errBoundTriangleFinal}.
\end{proof}

\begin{proof}[Proof of Proposition~\ref{prop:errOWNSP}]
    By the triangle inequality
    \[
    \|\sum_{k=0}^\infty(F_{++}V_{++}^{-1}V_{+-}F_{--}^{-2}
    V_{--}^{-1}V_{-+}F_{++})^k\|\leq \sum_{k=0}^\infty a^k
    \]
    for $a=\|F_{++}\|^2\|F_{--}^{-1}\|^2\|V_{++}^{-1}V_{+-}\|
    \|V_{--}^{-1}V_{-+}\|$, so that the Neumann series converges if $a<1$, or
    $\|F_{++}\|\|F_{--}^{-1}\|<(\|V_{++}^{-1}V_{+-}\|
    \|V_{--}^{-1}V_{-+}\|)^{-1/2}$. For $\|F_{++}\|\|F_{--}^{-1}\|<\epsilon$,
    with $\epsilon$ as in~\eqref{eq:errBoundEpsilon}, we obtain
    \[
        \|(I_{++}-F_{++}V_{++}^{-1}V_{+-}F_{--}^{-2}V_{--}^{-1}
        V_{-+}F_{++})^{-1}\|\leq1+\mathcal{O}(\epsilon^2).
    \]
    so that using~\eqref{eq:errBoundTriangleFinal} from
    Proposition~\eqref{prop:errBoundTriangle} yields
    \begin{equation}
        \|E-R_{N_\beta}ER_{N_\beta}^{-1}\|\leq
        \|F_{++}\|\|F_{--}^{-1}\|\big(\|V_{++}^{-1}V_{+-}\|+\|V_{--}^{-1}
        V_{-+}\|\big)+\mathcal{O}(\epsilon^2).
        \label{eq:errBoundNeumann}
    \end{equation}
    Finally, the induced Euclidean norm of a matrix is sub-multiplicative so
    that
    \[
    \|P-P_{N_\beta}\|\leq \|V\| \|E-R_{N_\beta}ER_{N_\beta}^{-1}\| \|V^{-1}\|,
    \]
    which yields~\eqref{eq:errOWNSP} when combined
    with~\eqref{eq:errBoundNeumann}.
\end{proof}

\bibliographystyle{siamplain}
\bibliography{ownslibrary}
\end{document}